\documentclass{amsart}
\usepackage[shortlabels]{enumitem}
\usepackage{xcolor}
\usepackage{stackengine}
\usepackage{dsfont}
\usepackage{enumitem}
\usepackage{physics}
\usepackage{amssymb}
\usepackage[T1]{fontenc}
\usepackage{amsmath}
\newtheorem{thm}{Theorem}[section]

\newtheorem{lem}[thm]{Lemma}

\newtheorem{Defi}[thm]{Definition}
\newtheorem{rk}[thm]{Remark}
\usepackage{mathtools}
\usepackage{fancyhdr}
\DeclarePairedDelimiter{\Vector}{\lparen}{\rparen}

\theoremstyle{definition}

\theoremstyle{remark}

\numberwithin{equation}{section}

\newcommand\shorttitle[1]{\renewcommand\@shorttitle{#1}}

\usepackage[babel,german=quotes]{csquotes}
\usepackage[backend=bibtex8,maxbibnames=99]{biblatex}
\bibliography{Texfile}

\begin{document}

%%
%% The title of the paper goes here.  Edit to your title.
%%

\title[Wild Solutions of the $3D$ axisymmetric Euler equations]{Wild Solutions of the $3D$ axisymmetric Euler equations}

%%
%% Now edit the following to give your name and address:
%% 

\author{Patrick Brkic}
\address{Institut für Angewandte Analysis, Universit\"at Ulm}
\email{patrick.brkic@uni-ulm.de}

\author{Emil Wiedemann}
\address{Department of Mathematics, Friedrich-Alexander-Universit\"at Erlangen-N\"urnberg}
\email{emil.wiedemann@fau.de}

\begin{abstract}
We consider the Cauchy problem for the $3D$ incompressible axisymmetric swirl-free Euler equations. The convex integration method developed by De Lellis and Sz\'{e}kelyhidi rules out the possibility that the Euler equations admit unique admissible weak solutions. It had remained conceivable, though, that axisymmetry of the solution might serve as a selection criterion. Using a surprising link to the $2D$ isentropic compressible Euler equations, we will show that this is not the case: There exists initial data for which there are infinetely many admissible swirl-free axisymmetric weak solutions of the $3D$ incompressible Euler equations. Moreover, somewhat conversely, we show that there exists an axisymmetric swirl-free initial velocity for which the axisymmetry breaks down instantaneously.
\end{abstract}

\maketitle

\section{Introduction}
The Euler equations constitute the fundamental model for ideal fluid flows, yet 
%They can formally be obtained as a vanishing viscosity limit of the Navier-Stokes equations, although they historically preceded the latter. Anyhow, the Euler equations are often viewed as an approximation of the Navier-Stokes equations when viscosity is small. This seems reasonable since real fluids are not perfectly inviscid. 
their anaylsis remains difficult and largely unresolved as soon as effects of turbulence come into play. 
%One of the main reasons for this lack is turbulence which is in particular barely understood from the mathematical but also from the physical point of view. Apart from that the mathematical theory of the three-dimensional Euler and Navier-Stokes equations also hold many interesting unresolved problems. 
A very prominent open problem is whether solutions with smooth initial data admit a global smooth solution. This problem is open for both the $3D$ incompressible Euler equations and Navier-Stokes equations \cite{Feffermann}, although remarkable progress has recently been made toward blow-up for the Euler equations~\cite{Elgindi, Hou1, Hou2}. Other problems, on which we focus here, are the existence and uniqueness of weak solutions to the Euler equations for given initial data. It is known that there are (many) initial conditions that give rise to non-unique solutions~\cite{Scheffer1993}, whereas the existence question is still open unless one surrenders energy admissibility~\cite{Wiedemann2011}.  

In contrast to the three-dimensional framework, the theory of the Euler and Navier-Stokes equations in two dimensions is much better understood. For this reason, $3D$ flows with two-dimensional character are important for the study of the three-dimensional setting. Symmetric flows are of particular interest (not least in the afore-mentioned studies on blow-up), and two symmetry classes are particularly prominent in the field of incompressible fluid dynamics: These are the classes of axisymmetric and of  helical flows, for which the respective symmetry binds one degree of freedom to effect an essentially $2D$ dynamics.

In this paper we deal with a class of symmetric flows modelled by the $3D$ incompressible axisymmetric swirl-free Euler equations. The model is obtained by a change to cylindrical coordinates of the Cauchy problem for the $3D$ incompressible Euler equations 
\begin{align}\label{incompressible 3D Euler equations}
\begin{cases}
&\partial_t v + \operatorname{div}(v\otimes v)+ \nabla \pi =0 \\
& \operatorname{div}(u)=0\\
& v(\cdot,0)=v^0,
\end{cases}
\end{align}
where $v$ is the velocity, $\pi$ is the pressure and the initial velocity satisfies $\operatorname{div}(v_0)=0$. Imposing that, in cylindrical coordinates, the velocity $$v=(v_r(r,\theta,z,t),v_\theta(r,\theta,z,t),v_z(r,\theta,z,t))$$
and the pressure $\pi$ only depend on the radial and the vertical variables $r$ and $z$ and the swirl component $v_{\theta}$ vanishes, we end up with the Cauchy problem for the axisymmetric swirl-free Euler equations in $(0,\infty)\times \mathbb{R}\times [0,T]$:
\begin{align}\label{Cauchy problem for the axisymmetric Euler equations}
\begin{cases}
&\partial_tv_r+\begin{pmatrix}
v_r\\v_z
\end{pmatrix}\cdot \begin{pmatrix}
\partial_r\\ \partial_z
\end{pmatrix}v_r+\partial_r \pi=0\\
&\partial_t v_z+\begin{pmatrix}
v_r\\v_z
\end{pmatrix}\cdot \begin{pmatrix}
\partial_r\\ \partial_z
\end{pmatrix}v_z+\partial_z \pi  =0\\
&\partial_r(rv_r)+\partial_z(rv_z) =0\\
&v(\cdot,0)=v^0.
\end{cases}
\end{align} 
Ukhovskii and Yudovich \cite{UKHOVSKII196852} were the first to give a global well-posedness result for the $3D$ incompressible Euler equations in the class of axisymmetric swirl-free weak solutions in $L^2_tH^1_x$ by imposing $v_0\in L^2\cap L^{\infty}$, $\omega_0\in L^2\cap L^{\infty}$, and $\xi_0\in L^2\cap L^{\infty}$, where $\omega_0$ is the initial vorticity and $\xi_0=\frac{\omega_0}{r}$ is the initial relative vorticity. Since then various existence and uniqueness results have been established, see e.g.~\cite{Abidi,Danchin_2007,Majda,Raymond,Chae,Chae2}. More recently, weak solutions of the axisymmetric swirl-free Euler equations were studied by means of the viscosity limit~\cite{JiuQuansen,nobili2019renormalization,WiedemannBrkic}. In contrast to~\eqref{incompressible 3D Euler equations}, it is well-known that in the axisymmetric swirl-free setting~\eqref{Cauchy problem for the axisymmetric Euler equations}, solutions are globally smooth if the initial velocity is smooth~\cite{Raymond,Majda}. For less regular initial data and related to the blow-up behavior of the $3D$ incompressible Euler equations, let us recall that in the recent work~\cite{Elgindi}, axisymmetric swirl-free finite time blow-up solutions for the vorticity formulation of~\eqref{incompressible 3D Euler equations} were constructed with infinite energy, and shortly later with finite energy~\cite{DrivasElgindi}.
Further, in~\cite{Elgindi} and~\cite{DrivasElgindi} it was conjectured that finite time blow-up seems to require velocity fields in $C^{1,\alpha}$ for $\alpha<\frac{1}{3}$. In~\cite{nev} the blow-up regime was investigated by means of different methods and the authors gave numerical evidence for $\frac{1}{3}$ as the critical threshold for finite time blow-up.

In this work we shall study very weak solutions of~\eqref{Cauchy problem for the axisymmetric Euler equations}. Since the groundbreaking works of De Lellis and Sz\'ekelyhidi \cite{convexintegration,convexintegration2}, showing non-uniqueness of weak solutions to the incompressible Euler equations for certain initial data, the method of convex integration has become a well-established tool to construct admissible weak solutions for the Euler equations. (By weak solution, we mean a distributional solution in the class $L^{\infty}_tL^2_x$, and we say it is admissible if it satisfies some kind of energy inequality.) As a consequence of their work for the incompressible Euler equations, they even concluded an analogous non-uniqueness result for admissible (entropy-)solutions of the isentropic compressible Euler equations
\begin{align}\label{compressible Euler equations}
\begin{cases}
&\partial_t \rho + \operatorname{div}(\rho v)=0\\
&\partial_t (\rho v) + \operatorname{div}(\rho v\otimes v)+ \nabla p(\rho)=0\\
&\rho(\cdot,0)=\rho^0\\
&v(\cdot,0)=v^0,
\end{cases}
\end{align}
where $v$ represents the velocity and $\rho$ the density of the gas, and $p$ is a function of the density $\rho$ and models the pressure.

Based on the ideas from \cite{convexintegration,convexintegration2}, Elisabetta Chiodaroli~\cite{Chiodaroli} designed a semi-stationary (i.e., the density is time-independent) convex integration scheme for~\eqref{compressible Euler equations} in terms of the momentum $m=\rho v$ in the space periodic setting:
\begin{align}\label{compressible Euler equations in terms of the momentum}
\begin{cases}
&\partial_t \rho + \operatorname{div}(m)=0\\
&\partial_t m + \operatorname{div}(\frac{m\otimes m}{\rho})+ \nabla p(\rho)=0\\
&\rho(\cdot,0)=\rho^0\\
&m(\cdot,0)=m^0.
\end{cases}
\end{align}
On bounded domains and the whole space, Akramov and Wiedemann~\cite{AkramovWiedemann} constructed admissible weak solutions with compact support.

Compared to the weak solutions studied in \cite{UKHOVSKII196852}, weak solutions constructed by convex integration are very weak. In the class of axisymmetric swirl-free velocity fields, the Euler equations~\eqref{incompressible 3D Euler equations} have not been studied by means of convex integration so far to our knowledge. Indeed, there are axisymmetric data giving rise to convex integration solutions~\cite{Scheffer1993,BardosSzWiedemann,Mengual}, but these are expected to break the symmetry instantaneously. 

 In this regard, we resolve two open problems for~\eqref{Cauchy problem for the axisymmetric Euler equations}:
\begin{itemize}
\item {\em Non-uniqueness under preservation of symmetry}: Does there exist axisymmetric swirl-free initial data $v^0\in L^2_{loc}(\mathbb{R}^3)$ for which there exist infinitely many admissible weak solutions $v$ of~\eqref{incompressible 3D Euler equations} preserving axisymmetry and the swirl-free property?
\item {\em Symmetry breaking}: Does there exist axisymmetric and swirl-free initial data $v^0\in L^2_{loc}(\mathbb{R}^3)$  and an admissible weak solution $v$ of~\eqref{incompressible 3D Euler equations} with $v(\cdot,0)=v^0$ for which $v(\cdot,t)$ breaks the symmetry for $t>0$?
\end{itemize}
The main goal of this work is to investigate these problems.
More precisely, we provide affirmative answers to the latter questions, at least for possibly small times: 
\begin{thm}\label{non-uniqueness under symmetry preservation introduction}
There exists $T>0$ and $v^0\in L^2_{loc}(\mathbb{R}^3;\mathbb{R}^3)$ axisymmetric and swirl-free for which there exist infinitely many weak solutions $v\in L^{\infty}(0,T;L^2_{loc}(\mathbb{R}^3;\mathbb{R}^3))$ of~\eqref{incompressible 3D Euler equations} which are axisymmetric and swirl-free almost everywhere.
\end{thm}
\begin{thm}\label{symmetry breaking introduction}
There exists $T>0$ and $v^0\in L^2_{loc}(\mathbb{R}^3;\mathbb{R}^3)$ axisymmetric and swirl-free and admissible weak solutions $v\in L^{\infty}(0,T;L^2_{loc}(\mathbb{R}^3;\mathbb{R}^3))$ of~\eqref{incompressible 3D Euler equations} for which $v(\cdot,t)$ is not axisymmetric almost everywhere for every $t\in(0,T)$. 
\end{thm}

It will be clear from the construction that both the data and the solutions obtained in both theorems can be chosen periodic in the vertical direction and can then be viewed as elements of $L^{\infty}(0,T;L^2(\mathbb{R}^2\times\mathbb{T};\mathbb{R}^3))$, where $\mathbb{T}=\mathbb{R}/\mathbb{Z}$ denotes the one-dimensional flat torus; we are therefore not `cheating' by working in $L^2_{loc}$, letting infinite energy intrude from spatial infinity or the like.  

Theorem \ref{non-uniqueness under symmetry preservation introduction} shows that we can construct many weak solutions of~\eqref{incompressible 3D Euler equations} which preserve axisymmetry and the no-swirl condition. It is the first result where convex integration is applied to the study of axisymmetric Euler equations. It can be compared to numerical simulations and experiments from axisymmetric turbulence (see~\cite{Ertunc2007} and references therein). 

%Theorem~\ref{symmetry breaking introduction}, on the other hand, is not a big surprise -- already the example of Scheffer~\cite{Scheffer1993} can be trivially lifted from two to three dimensions; it then starts from axisymmetric data (namely, zero) and then evolves in a way that is clearly non-axisymmetric. Our Theorem~\ref{symmetry breaking introduction} has the 

We will prove Theorem~\ref{non-uniqueness under symmetry preservation introduction} in Section 2. Let us outline our strategy and highlight the key novel key ingredients for this proof. In a first step we present an innovative link between~\eqref{Cauchy problem for the axisymmetric Euler equations} and~\eqref{compressible Euler equations}. In fact, for a specific choice of the density $\rho$ which only depends on the radial variable $r$, we will show that the $2D$ compressible Euler equations~\eqref{compressible Euler equations} in the variables $r,z$ are equivalent to the axisymmetric swirl-free Euler equations~\eqref{Cauchy problem for the axisymmetric Euler equations} on the halfplane $\mathbb{H}=(0,\infty)\times \mathbb{R}$. It will even turn out that the local energy inequalities are equivalent. This link (Lemma~\ref{equivalence of weak solutions for compressilbe Euler and axisymmetric Euler}) is one of the key novel ingredients and motivates us to study weak solutions of~\eqref{Cauchy problem for the axisymmetric Euler equations} with techniques developed for~\eqref{compressible Euler equations}. 

In a second step, we will study weak solutions of~\eqref{compressible Euler equations in terms of the momentum} via convex integration. For rather technical reasons, we will consider~\eqref{compressible Euler equations in terms of the momentum} on a strip $(\delta,R)\times \mathbb{R}$ of the halfplane $\mathbb{H}$. The second main ingredient will be the construction of a new suitable subsolution for~\eqref{compressible Euler equations in terms of the momentum}, Lemma~\ref{Konstruktion Sulsg}. In fact, by choosing a density only depending on $r$, the convex integration scheme enables us to construct subsolutions which are periodic in the vertical $z$ direction. This means in the construction scheme of subsolutions on the $r$-$z$ plane we will consider momentum fields which for any fixed $r$ point upwards in the $r$-$z$ plane. The construction of such subsolutions motivated us to study weak solutions on a strip $(\delta,R)\times \mathbb{R}$ of the halfplane $\mathbb{H}$ which are periodic in $z$ direction and satisfy a slip boundary condition. Finally we will use these weak solutions to construct admissible weak solutions for~\eqref{Cauchy problem for the axisymmetric Euler equations} in $\mathbb{H}$. By returning to Cartesian coordinates we end up with admissible weak solutions of~\eqref{incompressible 3D Euler equations} which are axisymmetric and swirl-free, cf.~Remark \ref{Remark weak solutions 3D and axisymmetric}. 

A proof of Theorem~\ref{symmetry breaking introduction} will be provided in Section 4. The key novel ingredient for this result is the construction of new subsolutions. The idea is to follow the construction of weak solutions for the $2D$ Euler equations for vortex sheet initial data~\cite{SZEKELYHIDI20111063} and rotational initial data~\cite{BardosSzWiedemann} and to lift it to the $3D$ axisymmetric swirl-free framework. More precisely, for a specific axisymmetric swirl-free initial datum $v^0$, we will reduce the existence of suitable subsolutions to the existence of a rarefaction solution of a Burgers equation. The existence of weak solutions then follows by convex integration. Moreover, by constructing a suitable energy profile, we will show that many admissible weak solutions will break the axisymmetry in the evolution.

A first example for symmetry breaking of the Euler equations goes back to Scheffer~{\cite{Scheffer1993}} (he starts with zero initial datum, which of course possesses any conceivable symmetry). For planar flows in three dimensions, symmetry-breaking has been investigated in~\cite{BardosNussenzveigTiti,Wiedemann}. Moreover, symmetry breaking has been considered in \cite{SZEKELYHIDI20111063} for the $2D$ incompressible Euler equations with vortex sheet initial data in the space periodic setting. Based on this, in~\cite{BardosSzWiedemann} a similar result was considered for the $2D$ incompressible Euler equations with rotational (i.e., pure-swirl) initial data. 
\section{Weak solutions of the $3D$ axisymmetric swirl-free Euler equations and the $2D$ isentropic compressible Euler equations}\label{Section Axisymmetric swirl-free Euler equations vs isentropic compressible Euler equations}
Let us begin by introducing the notion of weak solutions for (\ref{Cauchy problem for the axisymmetric Euler equations}) and (\ref{compressible Euler equations}):
\begin{Defi}
Let $v=v(r,\theta,z)$ be a vector field in cylindrical coordinates $(r,\theta,z)$ and 
\begin{align*}
e_r=\Vector{\cos(\theta),\sin(\theta),0}, \quad e_{\theta}=\Vector{-\sin(\theta),\cos(\theta),0},\quad e_z=\Vector{0,0,1}
\end{align*}
be the unit vectors in cylindrical coordinates.
\begin{enumerate}
\item[(i)] $v$ is called {\em axisymmetric} if $v$ has cylindrical symmetry in space, i.e., $v=v(r,z)$.
\item[(ii)] $v$ is called {\em swirl-free} if its angular component vanishes, i.e., $v_{\theta}=v\cdot e_\theta=0$.
\item[(iii)] $v$ is said to be axisymmetric or swirl-free {\em almost everywhere} if there exists an axisymmetric or swirl-free vector field $\tilde v$, respectively, such that $v=\tilde v$ almost everywhere.
\end{enumerate}
\end{Defi}
\begin{rk}
In the following we have to distinguish between the gradients in Cartesian and cylindrical coordinates. We will indicate the gradient in Cartesian coordinates by a subscript $x$. When we work in cylindrical coordinates we will use $\nabla$ to denote the gradient with respect to the variables $r$ and $z$, i.e. $\nabla = \begin{pmatrix}
\partial_r\\\partial_z
\end{pmatrix}$. Similarly $\operatorname{div}_x$ is the divergence in Cartesian coordinates and $\operatorname{div}$ denotes the divergence with respect to the variables $r,z$, that is: $\operatorname{div}v(r,z)=\partial_r v\cdot e_r+\partial_z v\cdot e_z$.

Note carefully that $\nabla$ and $\operatorname{div}$ are {\em not} the three dimensional gradient and divergence operators expressed in cylindrical coordinates; rather, they denote the gradient and divergence operators when $r,z$ are considered as $2D$ Cartesian coordinates, see also Remark~\ref{rzinterpret} below.
\end{rk}
\begin{Defi}
Let $v^0\in L^2_{loc}(\mathbb{R}^3;\mathbb{R}^3)$ be axisymmetric and swirl-free and ${\Omega}\subset \mathbb{H}=(0,\infty)\times\mathbb{R}$.
\begin{enumerate}[(i)]
\item We say that an almost everywhere axisymmetric and swirl-free vector field $v\in L^{\infty}(0,T;L^2_{loc}(\mathbb{R}^3;\mathbb{R}^3))$ is a weak solution of the axi\-symmetric swirl-free Euler equations (\ref{Cauchy problem for the axisymmetric Euler equations}) in ${\Omega}\times (0,T)$ if $v$ is weakly divergence-free, i.e.
\begin{align*}
\int_{\Omega} v(r,z,t)\cdot \nabla \varphi(r,z)r dz dr=0
\end{align*}
for every axisymmetric $\varphi\in C_c^{\infty}({\Omega}\times [0,T))$, for a.e. $t\in (0,T)$, and $v=v_r e_r+ v_z e_z$ satisfies
\begin{align}\label{weak formulation axisymmetric swirl-free Euler equations}
\int_0^T\int_{{\Omega}}&\partial_t \varphi r {v} + r\left(v_r^2\partial_r\varphi_r + v_rv_z\partial_z\varphi_r+v_rv_z \partial_r \varphi_z+v_z^2\partial_z\varphi_z\right) dz dr dt\\
&+\int_0^T\int_{{\Omega}}r\pi (\partial_r \varphi_r +\frac{\varphi_r}{r}+ \partial_z \varphi_z) dz dr dt + \int_{{\Omega}} r {v}^0\varphi(\cdot,0) dz dr=0\notag
\end{align}
for every axisymmetric $\varphi=\varphi_re_r+\varphi_z e_z\in C_c^{\infty}({\Omega}\times [0,T);\mathbb{R}^3)$.
\item We say that a weak solution $v$ of the axisymmetric swirl-free Euler equations is admissible if $v=v_r e_r + v_z e_z$ satisfies the local energy inequality 
\begin{align}\label{energy inequality axisymmetric swirl-free Euler eqations}
\frac{1}{2}\partial_t |{v}|^2r+ \operatorname{div}\left[\left(\frac{|{v}|^2}{2}+\pi\right)r{v}\right]\leq 0 \text{ in } \Omega
\end{align}
in the sense of distributions, i.e., we have
\begin{align*}
\int_0^T&\int_{{\Omega}} \frac{1}{2}\partial_t \varphi|{v}|^2r+ \left[\left(\frac{|{v}|^2}{2}+\pi\right)\right]r(v_r \partial_r + v_z \partial_z )\varphi dz dr dt\\
&+ \frac{1}{2}\int_{{\Omega}} \varphi(0)|{v}^0|^2 r dz dr\geq 0
\end{align*}
for every axisymmetric $\varphi\in C_c^{\infty}(\Omega\times [0,T))$.
\end{enumerate}
\end{Defi}
\begin{rk}\label{Remark weak solutions 3D and axisymmetric}
\begin{enumerate}[(i)]
\item The usual weak formulation of the $3D$ Euler equations and the formulation given here are equivalent in the swirl-free axisymmetric situation, that is: An almost everywhere swirl-free axisymmetric vector field is a weak solution of the $3D$ Euler equations if and only it fulfills Definition~\ref{weak formulation axisymmetric swirl-free Euler equations}. This can be seen by a coordinate transformation, using the expression for the divergence in cylindrical coordinates and the invariance of the Euclidean inner product and the tensor product of two vectors under coordinate transformation. Particularly, for the advection term, we obtain
\begin{equation}\label{polartransf}
\langle v\otimes v,\nabla_x\varphi\rangle=v_r^2\partial_r\varphi_r + v_rv_z\partial_z\varphi_r+v_rv_z \partial_r \varphi_z+v_z^2\partial_z\varphi_z.
\end{equation}
An explicit verification of this identity is given in the appendix.

\item A straightforward calculation shows that~\eqref{energy inequality axisymmetric swirl-free Euler eqations} is equivalent to the local energy inequality in $\tilde{\Omega}$, where $\tilde\Omega$ is the rotation of $\Omega$ about the $z$-axis, given by 
\begin{equation*}
\tilde\Omega=\{(r,\theta,z): (r,z)\in\Omega,\quad \theta\in\mathbb{R}\}\subset\mathbb{R}^3.
\end{equation*}
\end{enumerate}
\end{rk}
\begin{Defi}\label{defweakcomp}
Let ${\Omega}\subset \mathbb{R}^2$ and let $\rho^0\in L^{\gamma}_{loc}({\Omega})$. Let $v^0$ and $\rho^0$ be such that $\rho^0|v^0|^2\in L^1_{loc}({\Omega})$.
\begin{enumerate}[(i)]
\item We say that $(\rho,v)=(\rho(r,z),v(r,z))$ is a {\em weak solution} of the compressible Euler equations (\ref{compressible Euler equations}) in ${\Omega}\times (0,T)$ with pressure $p(\rho)=\rho^{\gamma}$ in ${\Omega}\times (0,T)$ if $\rho\in L^{\gamma}_{loc}({\Omega})$, $\rho|v|^2\in L^1_{loc}({\Omega}\times (0,T))$ and
\begin{align*}
\int_0^T\int_{{\Omega}}\partial_t \psi \rho v + \left\langle \rho v\otimes v, \nabla \psi\right\rangle  + \rho^{\gamma} \operatorname{div}(\psi) dz dr dt + \int_{{\Omega}} \rho^0 v(\cdot,0)\psi(\cdot,0) dz dr=0\\
\int_0^T\int_{{\Omega}} \rho \partial_t \varphi+ \rho v\cdot \nabla \varphi dz dr dt + \int_{{\Omega}} \rho^0 \varphi(\cdot,0) dz dr =0
\end{align*}
for all $\psi\in C_c^{\infty}({\Omega}\times [0,T);\mathbb{R}^2)$, $\varphi\in C_c^{\infty}({\Omega}\times [0,T))$.
\item We say that a weak solution of the compressible Euler equations~\eqref{compressible Euler equations} is {\em admissible} if it satisfies the local energy inequality 
\begin{align*}
\partial_t \left(\frac{\rho |v|^2}{2}+\frac{1}{\gamma -1}\rho^{\gamma}\right)+ \operatorname{div}\left[\left(\frac{\rho|v|^2}{2}+\frac{\gamma}{\gamma -1}\rho^{\gamma}\right)v\right]\leq 0 \text{ in } {\Omega}
\end{align*}
in the sense of distributions, i.e. 
\begin{align*}
\int_0^T&\int_{{\Omega}} \partial_t \varphi\left(\frac{\rho|v|^2}{2}+\frac{1}{\gamma -1}r^{\gamma}\right)+ \left[\left(\frac{|v|^2}{2}+\frac{\gamma}{\gamma -1}\rho^{\gamma}\right) v\right] \cdot \nabla \varphi dz dr dt\\
&+ \frac{1}{2}\int_{{\Omega}} \varphi(0)~\rho^0|v^0|^2 + \frac{1}{\gamma -1} \varphi(0)~(\rho^0)^{\gamma} dz dr\geq 0
\end{align*}
for every $\varphi\in C_c^{\infty}({\Omega}\times [0,T))$.
\end{enumerate}
\end{Defi}
\begin{rk}\label{rzinterpret}
Let us point out that the variables $r$, $z$ will be used in two different contexts. When we consider the $2D$ compressible Euler equations in velocity or in momentum we treat $r$ and $z$ as $2D$ variables and hence one can think of them as Cartesian coordinates. This unusual notation is justified by Lemma~\ref{equivalence of weak solutions for compressilbe Euler and axisymmetric Euler} below.

In contrast, $r$ and $z$ are cylindrical coordinates if we speak about the axisymmetric Euler equations. For the sake of notation we will write the components of $2D$ vector fields in Cartesian coordinates $r, z$ as $v=(v_r,v_z)$. When we talk about axisymmetric vector fields, we will indicate this by writing $v= v_r e_r + v_z e_z$, where  $e_r, e_{\theta}, e_z$ are the basis vectors in cylindrical coordinates.
\end{rk}
Let us now present an elementary but unexpected link between the $2D$ compressible and the $3D$ axisymmetric swirl-free Euler equations, which is key to our main result, Theorem~\ref{non-uniqueness under symmetry preservation introduction}. We consider the isentropic compressible $2D$ Euler equations (\ref{compressible Euler equations}) in the $r$-$z$ plane with polytropic pressure law $p(\rho)=\rho^{\gamma}$ for some constant $\gamma >1$. Consider the particular choice of density $\rho_0$, given independently of time as
\begin{align}\label{Dichte}
\rho_0(r,z,t)=\begin{cases}r &r>0\\
0 &r\leq0.
\end{cases}
\end{align}
Then we can reformulate the isentropic compressible Euler equations~\eqref{compressible Euler equations} on the plane-time cylinder $\mathbb{H}\times(0,T)$. In components we have
\begin{align*}
\begin{cases}
\partial_t(rv_r)+\partial_r(rv_r^2)+\partial_z(rv_zv_r)+\partial_r p(r,z)&=0\\
\partial_t(rv_z)+\partial_r(rv_rv_z)+\partial_z(rv_z^2)&=0\\
\partial_r(rv_r)+\partial_z(rv_z)&=0.
\end{cases}
\end{align*}
Due to the divergence condition this simplifies to
\begin{align*}
\begin{cases}
\partial_t(rv_r)+r\begin{pmatrix}
v_r\\v_z
\end{pmatrix}\cdot \begin{pmatrix}
\partial_r\\ \partial_z
\end{pmatrix}v_r+\partial_r p(r,z)&=0\\
\partial_t(rv_z)+r\begin{pmatrix}
v_r\\v_z
\end{pmatrix}\cdot \begin{pmatrix}
\partial_r\\ \partial_z
\end{pmatrix}v_z &=0\\
\partial_r(rv_r)+\partial_z(rv_z)&=0.
\end{cases}
\end{align*}
Dividing the first and second line by $r$ we end up with the axisymmetric swirl-free Euler equations
\begin{align}\label{axisymmetric Euler equations}
\begin{cases}
\partial_t v_r+ (v_r \partial_r + v_z \partial_z) v_r+\partial_r \pi(r,z)&=0\\
\partial_t v_z+(v_r \partial_r + v_z \partial_z)v_z &=0\\
\partial_r(rv_r)+\partial_z(rv_z)&=0
\end{cases}
\end{align} 
with pressure $\pi$,
\begin{align}\label{Definition 3D Druck}
\pi(r)=\int_0^r \frac{1}{s}\partial_s p(s)ds=\frac{\gamma}{\gamma -1}r^{\gamma-1}.
\end{align}
This calculation culminates in the following result:
\begin{lem}\label{equivalence of weak solutions for compressilbe Euler and axisymmetric Euler}
Let $v=(v_r,v_z)\in L^2(\Omega\times [0,T);\mathbb{R}^2)$ be a weak solution of the compressible $2D$ Euler equations (\ref{compressible Euler equations}) with density $\rho(r)=r$ and pressure $p(\rho)=\rho^{\gamma}$ for $\gamma >1$. Then $\tilde{v}=v_r e_r + v_z e_z$ is a weak solution of the $3D$ incompressible axisymmetric swirl-free Euler equations with pressure $\pi(r,z)=\frac{\gamma}{\gamma -1}r^{\gamma-1}$. Moreover if $v$ satisfies the local energy inequality for~\eqref{compressible Euler equations} then $\tilde v$ satisfies the local energy inequality for~\eqref{Cauchy problem for the axisymmetric Euler equations} with pressure $\pi(r)=\frac{\gamma}{\gamma -1}r^{\gamma-1}$.
\end{lem}
At the end of the section we will prove the remaining part of the lemma and show that the local energy inequalities are equivalent. Let us now recall the main result which is an ill-posedness result in the class of axisymmetric swirl-free weak solutions of~\eqref{incompressible 3D Euler equations} (cf. Theorem \ref{non-uniqueness under symmetry preservation introduction}).
\begin{thm}\label{main result}
There exist $v^0\in L^2_{loc}(\mathbb{R}^3)$ axisymmetric and swirl-free and $T>0$
%so that $rv_0\in L^{\infty}_{loc}(\mathbb{H})$. Then there exist
for which there exist infinitely many admissible weak solutions $v\in L^{\infty}(0,T;L^2_{loc}(\mathbb{R}^3;\mathbb{R}^3))$ of (\ref{incompressible 3D Euler equations}) which in cylindrical coordinates satisfy the axisymmetric swirl-free Euler equations (\ref{Cauchy problem for the axisymmetric Euler equations}) in $\mathbb{H}\times [0,T)$.
\end{thm}
This result will be proved at the end of the section. Before that, we state results for the existence of weak solutions of (\ref{compressible Euler equations in terms of the momentum}), which represent one of the building blocks of this work.
In a first step we give the definition of admissible weak solutions for~\eqref{compressible Euler equations in terms of the momentum} on a certain domain. For this purpose let $0<\delta<R$ and define
\begin{align*}
\Omega\coloneqq (\delta,R)\times \mathbb{T}
\end{align*}
where $\mathbb{T}=\mathbb{R}/\mathbb{Z}$ is the torus. Analogously we denote by $\Omega_{\mathbb{R}}$ the strip $\Omega_{\mathbb{R}}\coloneqq (\delta,R)\times \mathbb{R}$.
In order to formulate boundary values, we introduce the space of solenoidal momentum fields $H(\Omega;\mathbb{R}^2)$, which is defined as the completion of
\begin{align}\label{solenoidal vector fields}
\{m\in C_c^{\infty}(\Omega;\mathbb{R}^2):\operatorname{div}(m)=0\}
\end{align}
with respect to the $L^2(\Omega;\mathbb{R}^2)$ topology. In order to keep the notation short we will use $\|\cdot\|_{L^2(\Omega)}$ instead of $\|\cdot\|_{L^2(\Omega;\mathbb{R}^2)}$. With $H_w(\Omega;\mathbb{R}^2)$ we denote the same space endowed with the weak topology.
Thanks to this definition any $m\in H(\Omega;\mathbb{R}^2)$ is incompressible, and in the sense of traces \cite[Theorem 1.2.~ and Remark 1.3.]{Temam} $m$ satisfies the slip boundary condition $m\cdot n=0$ in $H^{-\frac{1}{2}}(\Gamma)$ where $n$ is the outer unit normal to the boundary $\Gamma$ of $\Omega$. In fact, this means $m\in H(\Omega;\mathbb{R}^2)$ is periodic in $z$ direction and fulfills $m_r=0$ on $\{r=\delta\}\cup \{r=R\}$ in the trace sense just mentioned. 
%
%
%\begin{Defi}
%Let $\rho \in L^{\infty}(\Omega\times (0,T))$, $m\in L^{\infty}(0,T;H(\Omega;\mathbb{R}^2))$. Then 
%%
%\begin{enumerate}[(i)]
%\item $(\rho,m)$ is a {\em weak solution} of (\ref{compressible Euler equations in terms of the momentum}) on $\Omega\times [0,T)$ if 
%\begin{align*}
%\int_0^T\int_{\Omega}m\partial_t \varphi + \left\langle \frac{m\otimes m}{\rho}, \nabla \varphi\right\rangle + p(\rho) \operatorname{div} \varphi ~dz dr dt +\int_{\Omega} m_0 \varphi(r,z,0) ~dz dr=0,\\
%\int_0^T\int_{\Omega}\rho \partial_t \psi + m \cdot \nabla \psi~ dz dr dt + \int_{\Omega} \rho_0(r,z) \psi(r,z,0)~dz dr =0
%\end{align*}
%for all $\varphi\in C_c^{\infty}([0,T)\times \Omega;\mathbb{R}^2)$.
%\item A weak solution $(\rho,m)$ is {\em admissible} if it satisfies the admissibility condition
%\begin{align*}
%&\int_0^T\int_{\Omega} \left(\rho\varepsilon(\rho)+\frac{1}{2}\frac{|m|^2}{\rho}\right)\partial_t \varphi + \left(\varepsilon(\rho)+\frac{1}{2}\frac{|m|^2}{\rho^2}+\frac{p(\rho)}{\rho}\right)m\cdot \nabla \varphi dx dt\\
%&+ \int_{\Omega} \left(\rho_0\varepsilon(\rho_0)+\frac{1}{2}\frac{|m_0|^2}{\rho}\right)\varphi(\cdot,0) dx \geq 0
%\end{align*}
%for all $\varphi\in C_c^{\infty}([0,T)\times \Omega)$, $\varphi\geq 0$.
%\end{enumerate}
%\end{Defi}
%
%
Now, we state an existence and non-uniqueness result for weak solutions of the compressible Euler equations in momentum formulation~\eqref{compressible Euler equations in terms of the momentum}. Of course, a weak solution of this system is a pair $(\rho,m)$ such that $(\rho,v)$ is a weak solution in the sense of Definition~\ref{defweakcomp} where $v=\frac m\rho$. (The density constructed in this paper is bounded away from zero, so that the denominator poses no issue.) 

In contrast to to the construction of weak solutions in \cite{AkramovWiedemann}, where the authors consider weak solutions with compact support, in the following two theorems we will investigate weak solutions in $\Omega$ with mixed boundary conditions, these are periodic boundary conditions in $z$ direction and slip boundary conditions on $\{r=\delta\}\cup \{r=R\}$ in the above-mentioned trace sense. We will provide proofs and the construction in Section 3.
\begin{thm}\label{non-uniqueness for compressible Euler equations}
Let $T>0$, $\rho_0$ be as in (\ref{Dichte}), and $p=|\cdot|^{\gamma}$ for a $\gamma >1$. Then for any $0<\delta<1<R$, setting $\Omega=(\delta,R)\times \mathbb{T}$, there exists $m^0\in (L^{\infty}\cap H)(\Omega;\mathbb{R}^2)$ so that there are infinitely many admissible weak solutions $(\rho,m)\in C^1(\Omega)\times C([0,T],H_w(\Omega;\mathbb{R}^2))$ of 
\begin{align}\label{Sublsg kompressible Euler}
\begin{cases}
&\partial_t m + \operatorname{div}(\frac{m\otimes m}{\rho})+ \nabla p(\rho)=0\\
&\partial_t \rho + \operatorname{div}(m)=0\\
&m(\cdot,0)=m^0\\
&m_r=0 \text{ on } \{r=\delta\}\cup \{r=R\}
\end{cases}
\end{align}
on $\Omega\times [0,T]$ with density $\rho(r,z)=\rho_0(r,z)$. Moreover, the weak solutions satisfy
\begin{align}\label{pointwise constraint subsolutions compressible Euler}
\begin{cases}
|m(r,z,t)|^2=\rho_0(r,z)\chi(t) &\text{ a.e. in } \Omega\times [0,T)\\
|m^0(r,z)|^2=\rho_0(r,z)\chi(0) &\text{ a.e. in } \Omega
\end{cases}
\end{align}
for some $\chi \in C^{\infty}([0,T];\mathbb{R}^+)$ to be fixed in the construction.
\end{thm}
From the weak solutions exhibited in Theorem \ref{non-uniqueness for compressible Euler equations}, we can select weak solutions which even satisfy the local energy inequality, at least up to a small time. This is the content of the following Theorem. We use $\varepsilon=\varepsilon(\rho)$ to denote the internal energy determined by $p(\rho)=\rho^2\varepsilon'(\rho)$.
\begin{thm}\label{non-uniqueness for compressible Euler equations with energy inequality}
Under the assumptions of Theorem~\ref{non-uniqueness for compressible Euler equations}, there exist $T>0$ and $m^0\in (L^{\infty}\cap H)(\Omega;\mathbb{R}^2)$ and infinitely many admissible weak solutions $(\rho,m)\in C^1(\Omega)\times C([0,T],H_w(\Omega;\mathbb{R}^2))$ fulfilling (\ref{Sublsg kompressible Euler}), (\ref{pointwise constraint subsolutions compressible Euler}) and $\rho=\rho_0$ which satisfy the admissibility condition
\begin{align}\label{local energy inequality for incompressible Euler equations momentum}
&\int_0^T\int_{\Omega} \left(\rho\varepsilon(\rho)+\frac{1}{2}\frac{|m|^2}{\rho}\right)\partial_t \varphi + \left(\varepsilon(\rho)+\frac{1}{2}\frac{|m|^2}{\rho^2}+\frac{p(\rho)}{\rho}\right)m\cdot \nabla \varphi dx dt\\
&+ \int_{\Omega} \left(\rho_0\varepsilon(\rho_0)+\frac{1}{2}\frac{|m_0|^2}{\rho}\right)\varphi(\cdot,0) dx \geq 0\notag
\end{align}
for all $\varphi\in C_c^{\infty}(\Omega\times [0,T);\mathbb{R}^2)$.
\end{thm}
\begin{proof}[Proof of Lemma~\ref{equivalence of weak solutions for compressilbe Euler and axisymmetric Euler}]
We only have to show that if a weak solution of (\ref{compressible Euler equations}) satisfies the energy inequality it already satisfies the energy inequality for (\ref{Cauchy problem for the axisymmetric Euler equations}).
Now, let $v=(v_r,v_z)$ be a weak solution of (\ref{compressible Euler equations}) which fulfils the local energy inequality for (\ref{compressible Euler equations}) in $\Omega$. Let $\varphi\in C_c^{\infty}({\Omega}\times [0,T))$, $\varphi \geq 0$. Since density $\rho$ and pressure $p$ are given, we have $\varepsilon(r)=\frac{1}{\gamma -1}r^{\gamma -1}$. Then the local energy inequality for (\ref{compressible Euler equations}) is equivalent to
\begin{align*}
&\int_0^T\int_{{\Omega}} \left(r\frac{1}{\gamma -1}r^{\gamma -1} +\frac{1}{2}r|{v}|^2\right)\partial_t \varphi + \left(\frac{1}{2}|{v}|^2+\frac{\gamma}{\gamma -1}r^{\gamma-1}\right)r{v}\cdot \nabla \varphi ~dz dr dt\\
&+ \int_{{\Omega}} \left(r \frac{1}{\gamma -1}r^{\gamma -1} +\frac{1}{2} r|{v}^0|^2\right)\varphi(\cdot,0)~dz dr \geq 0
\end{align*}
for all $\varphi\in C_c^{\infty}(\Omega\times [0,T))$. Integration by parts leads to
\begin{align*}
\int_0^T&\int_{{\Omega}} \frac{1}{2}r|{v}|^2 \partial_t \varphi + \left(\frac{\gamma}{\gamma -1}r^{\gamma -1}+\frac{1}{2}|{v}|^2\right)r{v}\cdot \nabla \varphi ~dz dr dt\\
&+ \int_{{\Omega}} \frac{1}{2} r|{v}^0|^2\varphi(\cdot,0)~dz dr \geq 0,
\end{align*}
which is the local energy inequality (\ref{energy inequality axisymmetric swirl-free Euler eqations}) for the axisymmetric swirl-free Euler equations (\ref{Cauchy problem for the axisymmetric Euler equations}) with pressure $\pi(r)=\frac{\gamma}{\gamma -1}r^{\gamma-1}$.
\end{proof}
\begin{proof}[Proof of Theorem \ref{main result}]
Let $\rho_0$ be as in (\ref{Dichte}) and let $p(\rho_0)=\rho_0^{\gamma}$ for $\gamma >1$. Due to Theorem \ref{non-uniqueness for compressible Euler equations with energy inequality} there exist $T>0$ and $m^0\in L^{\infty}(\Omega,\mathbb{R}^2)$ and infinitely many admissible weak solutions $(\rho,m)\in C^1(\Omega)\times C([0,T],H_w(\Omega;\mathbb{R}^2))$ with $\rho=\rho_0$ fulfilling (\ref{Sublsg kompressible Euler}) and (\ref{pointwise constraint subsolutions compressible Euler}) which satisfy the local energy inequality~\eqref{local energy inequality for incompressible Euler equations momentum}. Since $m\in C([0,T],H_w(\Omega;\mathbb{R}^2))$ and $r\in (\delta,R)$, we conclude that $v=\frac{m}{r}\in C([0,T],L^2_w(\Omega;\mathbb{R}^2))$, $v_0=\frac{m^0}{r}$ solves~\eqref{compressible Euler equations} in $\Omega\times (0,T)$ in the distributional sense. By Theorem \ref{equivalence of weak solutions for compressilbe Euler and axisymmetric Euler} we know that $\tilde{v}=v_r e_r + v_z e_z$ is a weak solution of (\ref{incompressible 3D Euler equations}) which is axisymmetric and swirl-free with pressure $\pi(r)=\frac{\gamma}{\gamma -1}r^{\gamma-1}$.
Moreover, $\tilde v\in C([0,T];H_w(\tilde\Omega);\mathbb{R}^3)$, and also Theorem \ref{equivalence of weak solutions for compressilbe Euler and axisymmetric Euler} provides the equivalence of the energy inequalities locally in $\Omega_{\mathbb{R}}$. 

Now, as $\tilde v$ is solenoidal, we may extend it by $\tilde{v}_0=0$, $\tilde{v}=0$ and $\pi$ constant outside $\Omega_{\mathbb{R}}$, so that we end up with a weak solution $\tilde{v}\in L^{\infty}((0,T),L^2_{loc}(\mathbb{R}^3;\mathbb{R}^3))$ of (\ref{Cauchy problem for the axisymmetric Euler equations}).
\end{proof}
\section{Convex integration and suitable subsolutions for the isentropic compressible Euler equations}
In the seminal works \cite{convexintegration,convexintegration2}, De Lellis and Székelyhidi used Gromov's convex integration theory together with Tartar's framework of plane wave analysis to study the incompressible Euler equations as a differential inclusion. They also obtained the first examples of convex integration solutions for the isentropic Euler system. Chiodaroli~\cite{Chiodaroli} then refined the method for the isentropic compressible Euler equations~\eqref{compressible Euler equations in terms of the momentum} to obtain a larger set of `wild' initial data (i.e., initial data that gives rise to infinitely many admissible solutions).

In this section, we recall the main points of the convex integration scheme developed for weak solutions of~\eqref{compressible Euler equations in terms of the momentum} in the whole space and bounded domains~\cite{AkramovWiedemann}. In  the same spirit we construct weak solutions of~\eqref{compressible Euler equations in terms of the momentum} in $\Omega\times (0,T)$ which will be periodic in $z$ direction and fulfill a slip boundary condition. However, instead of working in Fourier space as in~\cite{Chiodaroli}, we rather adapt the techniques of \cite{AkramovWiedemann} to the case of weak solutions which are not compactly supported but periodic in the $z$ direction. We consider
\begin{align}\label{compressible Euler equations in terms of the momentum on the strip}
\begin{cases}
\partial_t \rho + \operatorname{div}(m)=0 \quad&\text{in $\Omega_{\mathbb{R}}\times [0,T]$}\\
\partial_t m + \operatorname{div}(\frac{m\otimes m}{\rho})+ \nabla p(\rho)=0 \quad&\text{in $\Omega_{\mathbb{R}}\times [0,T]$} \\
\rho(\cdot,0)=\rho^0\quad& \text{in $\Omega_{\mathbb{R}}$}\\
m(\cdot,0)=m^0 \quad&\text{in $\Omega_{\mathbb{R}}$}\\
m_r=0 \quad&\text{on $\{r=\delta\}\cup \{r=R\}$}.
\end{cases}
\end{align}
Let $S_0^2$ be the set of symmetric $2\times 2$ matrices with zero trace.
The upcoming lemma is a reformulation of a result from \cite{convexintegration,AkramovWiedemann} on bounded domains which relates the Euler equations to a differential inclusion, that is a linear system of partial differential equations and a nonlinear constraint expressed by (\ref{Sublsg system}) and (\ref{Sublsg Bedingung}) below:
\begin{lem}\label{subsolutions to weak solutions}
Let $m\in L^{\infty}((0,T),(L^{\infty}\cap H)(\Omega\times (0,T);\mathbb{R}^2))$, $U\in L^{\infty}(\Omega\times (0,T);S_0^2)$ and $q\in L^{\infty}(\Omega\times (0,T))$ such that
\begin{align}\label{Sublsg system}
\begin{cases}
\partial_t m +\operatorname{div}(U)+ \nabla_x q =0 \quad&\text{in $\Omega_{\mathbb{R}}\times [0,T]$}\\
\operatorname{div}(m)=0.&
\end{cases}
\end{align}
If $(m,U,q)$ solve (\ref{Sublsg system}) and there exists $\rho\in L^{\infty}(\Omega,\mathbb{R}_+)$ so that
\begin{align}\label{Sublsg Bedingung}
\begin{cases}
U=\frac{m\otimes m}{\rho}-\frac{|m|^2}{2\rho}I_2 &\text{a.e. in $\Omega_{\mathbb{R}}\times [0,T]$},\\
q=p(\rho)+\frac{|m|^2}{2\rho} &\text{a.e. in $\Omega_{\mathbb{R}}\times [0,T]$},
\end{cases}
\end{align}
then $m$ and $\rho$ solve (\ref{compressible Euler equations in terms of the momentum on the strip} ) distributionally in $\Omega_{\mathbb{R}}\times [0,T]$.

Conversely, if $(m,\rho)$ is a weak solution of (\ref{compressible Euler equations in terms of the momentum on the strip}) in $\Omega_{\mathbb{R}}\times [0,T]$ then $m$, $U=\frac{m\otimes m}{\rho}-\frac{|m|^2}{2\rho}I_2$, and $q=p(\rho)+\frac{|m|^2}{2\rho}$ satisfy (\ref{Sublsg system}) and (\ref{Sublsg Bedingung}).
\end{lem}
In this section we will seek weak solutions with the specific density given by~\eqref{Dichte} and momentum in the space of solenoidal vector fields $H(\Omega;\mathbb{R}^2)$. 

Following \cite{convexintegration,convexintegration2} the differential inclusion~\eqref{Sublsg system} and~\eqref{Sublsg Bedingung} can be analyzed by introducing the wave cone $\Lambda$, which is the set of states determining plane wave solutions of~\eqref{Sublsg system}, and a closed constraint set $K$ which incorporates the nonlinear constraint~\eqref{Sublsg Bedingung}. For this we consider the symmetric matrix
\begin{align*}
M=\begin{pmatrix}
U+qI_2 & m\\
m & 0
\end{pmatrix}
\end{align*}
and define the wave cone $\Lambda$ to be
\begin{align*}
\Lambda\coloneqq \{(m,U,q)|~\exists \xi \in \mathbb{R}^3\setminus \{0\}: (m,U,q)h(x\cdot\xi) \text{ solves } (\ref{Sublsg system}) \text{ for every } h\colon \mathbb{R}\to \mathbb{R}\}.
\end{align*}
As a very brief computation shows, the wave cone can equivalently be characterized by
\begin{align*}
\Lambda= \{(m,U,q)\in \mathbb{R}^2\times S_0^2\times \mathbb{R}:~\det(M)=0\}.
\end{align*}

As for the nonlinear constraint~\eqref{Sublsg Bedingung} we consider
\begin{align*}
K_{\rho}\coloneqq\left\{(m,U,q)\in \mathbb{R}^2\times S_0^2\times \mathbb{R}^+:~U=\frac{m\otimes m}{\rho}-\frac{|m|^2}{2\rho}I_2,~q=p(\rho)+\frac{|m|^2}{2\rho}\right\}
\end{align*}
for $\rho\in (0,\infty)$. For such $\rho$ and for $\chi\in \mathbb{R}^+$, we set
\begin{align*}
K_{\rho,\chi}\coloneqq K_{\rho}\cap\{|m|^2=\rho\chi\}.
\end{align*}
Now, in the following we recall important characterizations for $K_{\rho,\chi}$ from \cite{Chiodaroli}, based on~\cite{convexintegration2}, and introduce a natural energy profile for~\eqref{Sublsg system}:
\begin{lem}\label{properties of the energy}
For $(\rho,m,U)\in \mathbb{R}^+\times \mathbb{R}^2\times S_0^2$ let
\begin{align*}
e(\rho,m,U)\coloneqq \lambda_{\max}\left(\frac{m\otimes m}{\rho}-U\right),
\end{align*}
where $\lambda_{\max}(A)$ denotes the largest eigenvalue of $A\in S_0^2$.
Then we have
\begin{enumerate}[(i)]
\item $e(\rho,\cdot,\cdot)\colon \mathbb{R}^2\times S_0^2\to \mathbb{R}$ is convex,
\item $\frac{|m|^2}{2\rho}\leq e(\rho,m,U)$ with equality if and only if $U=\frac{m\otimes m}{\rho}-\frac{|m|^2}{2\rho}$,
\item $\|U\|\leq e(\rho,m,U)$, where $\|\cdot\|$ denotes the operator norm,
\item $K_{\rho,\chi}^{co}=\{(m,U,q)\in \mathbb{R}^2\times S_0^2\times \mathbb{R}^+:~e(\rho,m,U)\leq \frac{\chi}{n},\quad q=p(\rho)+\frac{\chi}{n}\}$,
\item $K_{\rho,\chi}=K_{\rho,\chi}^{co}\cap \{|m|^2=\rho\chi\}$.
\end{enumerate}
\end{lem}
Here, $K^{co}_{\rho,\chi}$ denotes the convex hull of $K_{\rho,\chi}$. Chiodaroli \cite{Chiodaroli} pointed out that in contrast to the incompressible case \cite{convexintegration}, the interior of the set $K_{\rho,\chi}^{\text{co}}$ encoding the nonlinear constraint is empty, as it is a subset of the hyperplane $H=\{(m,U,q)\in \mathbb{R}^2\times S_0^2\times \mathbb{R}^+:~q=p(\rho)+\frac{\chi}{n}\}$. For this reason, Chiodaroli considers the {\em hyperinterior}:
For $\rho,\chi\in \mathbb{R}^+$ we define
\begin{align*}
\text{hint }K_{\rho,\chi}^{co}\coloneqq\left\{(m,U,q)\in \mathbb{R}^2\times S_0^2\times \mathbb{R}^+:~e(\rho,m,U)< \frac{\chi}{n}, q=p(\rho)+\frac{\chi}{n}\right\}.
\end{align*}
It was concluded that if the hyperinterior is nonempty, then the convex hull is large in some sense~\cite[Section3]{Chiodaroli}, which is necessary in order to study~\eqref{compressible Euler equations in terms of the momentum on the strip} by Tartar's plane wave analysis.

Now, we give the definition of subsolutions for~\eqref{compressible Euler equations in terms of the momentum on the strip} in $\Omega\times (0,T)$, for the specific density chosen in~\eqref{Dichte}:
\begin{Defi}[Subsolutions]\label{definition subsolutions}
Let $\rho_0\colon \Omega_{\mathbb{R}}\to \mathbb{R}$, $\rho_0(r,z)=r$ and $p(\rho_0)\colon \Omega_{\mathbb{R}}\to \mathbb{R}$, $p(\rho)=\rho^{\gamma}$ for $\gamma >1$. Let $(m_0,U_0,q_0)\colon \Omega_{\mathbb{R}}\times (0,T)\to \mathbb{R}^2\times S_0^2\times \mathbb{R}$ with $m_0\in C([0,T],H_w(\Omega;\mathbb{R}^2))$, $U_0\in C(\Omega_{\mathbb{R}}\times (0,T);S_0^2)$ and 
\begin{align*}
q_0(r,z,t)=p(\rho_0(r,z,t))+\frac{\chi(t)}{2} \text{ for all } (r,z,t)\in \Omega_{\mathbb{R}}\times (0,T)
\end{align*}
for given $\chi\in C^{\infty}([0,T],\mathbb{R}^+)$ such that
\begin{align*}
\begin{cases}
\partial_t m_0+\operatorname{div}(U_0)+\nabla q_0=0 &\text{ in } \Omega_{\mathbb{R}}\times (0,T)\\
\operatorname{div}(m_0)=0&\text{ in } \Omega_{\mathbb{R}}\times (0,T)\\
(m_0)_r=0 &\text{ on } \{r=\delta\}\cup \{r=R\}
\end{cases}
\end{align*}
and
\begin{align*}
e(\rho_0(r,z),m_0(r,z,t),U_0(r,z,t))<\frac{\chi(t)}{2}
\end{align*}
for all $(r,z,t)\in \Omega_{\mathbb{R}}\times [0,T]$.

A {\em subsolution} to (\ref{compressible Euler equations in terms of the momentum on the strip}) with respect to $\chi$ and $\rho_0$ is a continuous triple $(m,U,q)\colon \Omega_{\mathbb{R}}\times (0,T)\to \mathbb{R}^2\times S_0^2\times \mathbb{R}$ so that
\begin{align}\label{Sublsg erste Bedingung}
\begin{cases}
\partial_t m+\operatorname{div}(U)+\nabla q=0 &\text{ in }\Omega_{\mathbb{R}}\times (0,T)\\
\operatorname{div}(m)=0 &\text{ in }\Omega_{\mathbb{R}}\times (0,T)\\
m_r=0 &\text{ on } \{r=\delta\}\cup \{r=R\}
\end{cases}
\end{align}
and
\begin{align}\label{Sublsg zweite Bedingung}
\begin{cases}
q=q_0&\text{ for all }(r,z,t)\in \Omega_{\mathbb{R}}\times (0,T)\\
e(\rho_0(r,z),m(r,z,t),U(r,z,t))<\frac{\chi(t)}{2} &\text{ for all }(r,z,t)\in \Omega_{\mathbb{R}}\times (0,T)\\
m(r,z,0)=m_0(r,z,0)&\text{ for all }(r,z)\in \Omega_{\mathbb{R}}\\
m(r,z,T)=m_0(r,z,T)&\text{ for all }(r,z)\in \Omega_{\mathbb{R}}.
\end{cases}
\end{align}
\end{Defi}
In this context we define $X_0$ to be the set of all admissible momentum fields 
\begin{align*}
X_0\coloneqq \{m\in C((0,T),C(\Omega))\cap C([0,T],H_w(\Omega)): m\text{ fulfils } (\ref{Sublsg erste Bedingung}), (\ref{Sublsg zweite Bedingung}) \text{ for some } U\}.
\end{align*}
Note that for any $m\in X_0$, by Lemma~\ref{properties of the energy}(ii) and~\eqref{Sublsg zweite Bedingung} we have
\begin{align*}
|m|^2\leq  2 \rho_0 e(\rho_0,m,U)\leq  \chi \rho_0\quad\text{a.e.\ in $\Omega\times (0,T)$.}
\end{align*}
Due to the periodicity, this implies $|m|^2\leq \chi\rho_0$ a.e.\ in $\Omega_{\mathbb{R}}\times (0,T)$.
Now, let $X$ be the closure of $X_0$ with respect to the $C([0,T],H_w(\Omega;\mathbb{R}^2))$-norm.
The next lemma relates subsolutions to weak solutions which are periodic in $z$ direction and satisfy a slip boundary condition on $\{r=\delta\}\cup \{r=R\}$, in the trace sense discussed above. 
\begin{lem}[Subsolution criterion]\label{Subsolution criterion}
Let $\rho_0$ and $p(\rho_0)$ be given as in Definition \ref{definition subsolutions}. Assume there exist $(m_0,U_0,q_0)\colon \Omega_{\mathbb{R}}\times (0,T)\to \mathbb{R}^2\times S_0^2\times \mathbb{R}$ and $\chi\in C^{\infty}([0,T],\mathbb{R}^+)$ so that $m_0\in C([0,T],H_w(\Omega;\mathbb{R}^2))$, $U_0\in C(\Omega_{\mathbb{R}}\times (0,T);S_0^2)$ and 
\begin{align*}
q_0(r,z,t)=p(\rho_0(r,z,t))+\frac{\chi(t)}{2} \text{ for all } (r,z,t)\in \Omega\times (0,T)
\end{align*}
for given $\chi\in C^{\infty}([0,T],\mathbb{R}^+)$ fulfilling
\begin{align*}
\begin{cases}
\partial_t m_0+\operatorname{div}(U_0)+\nabla q_0=0 &\text{ in } \Omega_{\mathbb{R}}\times (0,T)\\
\operatorname{div}(m_0)=0&\text{ in } \Omega_{\mathbb{R}}\times (0,T)\\
(m_0)_r=0 &\text{ on } \{r=\delta\}\cup \{r=R\}\\
e(\rho_0(r,z),m_0(r,z,t),U_0(r,z,t))<\frac{\chi(t)}{2} &\text{ for all } (r,z,t)\in \Omega_{\mathbb{R}}\times (0,T).
\end{cases}
\end{align*}
Then there exist infinitely many weak solutions $(\rho,m)$ of (\ref{compressible Euler equations in terms of the momentum on the strip} ) in $\Omega_{\mathbb{R}}\times [0,T)$ with density $\rho(r,z)=\rho_0(r,z)=r$ so that
\begin{align*}
\begin{cases}
m\in C([0,T],H_w(\Omega;\mathbb{R}^2))\\
m(r,z,t)=m_0(r,z,t) &\text{ for } t=0,T \text{ and for a.e. } (r,z)\in \Omega_{\mathbb{R}}\\
|m(r,z,t)|^2=\rho_0(r,z)\chi(t)&\text{ for a.e. } (r,z,t)\in \Omega_{\mathbb{R}}\times (0,T)\\
m_r=0 &\text{ on } \{r=\delta\}\cup \{r=R\}.
\end{cases}
\end{align*}
\end{lem}
In order to prove Lemma \ref{Subsolution criterion}, three important tools are presented. Their proofs are almost identical to their counterparts in~\cite{convexintegration2, Chiodaroli}, but we present them in the appendix for the readers' convenience.

 Let us now start with the first ingredient which is a sufficient condition for  admissible momentum fields $m\in X$ to be weak solutions of~\eqref{compressible Euler equations in terms of the momentum on the strip}. 
\begin{lem}\label{weak solution condition}
If $m\in X$ fulfils $|m(r,z,t)|^2=\rho_0(r,z)\chi(t)$ for a.e. $(r,z,t)\in \Omega_{\mathbb{R}}\times (0,T)$ then the pair $(\rho_0,m)$ is a weak solution of~\eqref{compressible Euler equations in terms of the momentum on the strip} in $\Omega\times (0,T)$.
\end{lem}

The second ingredient is a characterization of the closure $X$ of admissible subsolutions $X_0$ in terms of Baire theory.
\begin{lem}\label{bairelem}
The identity map $I\colon (X,d)\to L^2([0,T],H(\Omega))$ defined by $m\mapsto m$ is a Baire-1 map and therefore the set of points of continuity is residual in $(X,d)$.
\end{lem}
The last ingredient is the so-called perturbation property. It is a cornerstone for the construction of weak solutions for~\eqref{compressible Euler equations in terms of the momentum on the strip}.
\begin{lem}\label{perturbation property}
Let $\rho_0,\chi$ be given as in Lemma \ref{Subsolution criterion}. Then there exists a constant $\beta >0$ so that given $m\in X_0$, there exists a sequence $(m_k)_{k\in \mathbb{N}}\subset X_0$ with
\begin{align*}
\|m_k\|_{L^2(\Omega\times [0,T])}^2\geq \|m\|_{L^2(\Omega\times [0,T])}^2 + \beta \left(\int_{\Omega}\int_0^T \rho_0(r,z)\chi(t) dt dz dr - \|m\|_{L^2(\Omega\times [0,T])}^2\right)^2
\end{align*}
and $m_k\to m$ in $C([0,T],H_w(\Omega;\mathbb{R}^2))$.
\end{lem}
\begin{proof}
The proof can be consulted from \cite[Lemma 4.5]{Chiodaroli}.
\end{proof}
\begin{proof}[Proof of Lemma \ref{Subsolution criterion}]
We show that all points of continuity of the identity map $I\colon (X,d)\to L^2([0,T],H(\Omega))$ correspond to solutions of (\ref{compressible Euler equations in terms of the momentum on the strip}) satisfying the conditions of Lemma \ref{Subsolution criterion}. To this end, we prove that if $m\in X$ is a point of continuity of $I$ we have
\begin{align*}
|m(r,z,t)|^2=\rho_0(r,z)\chi(t) \text{ for a.e. } (r,z,t)\in \Omega\times (0,T).
\end{align*}
As in \cite{Chiodaroli} we only have to show that
\begin{align*}
\|m\|_{L^2(\Omega\times (0,T))} \geq \left(\int_{\Omega}\int_0^T\rho_0(r,z)\chi(t) dt dz dr\right)^{\frac{1}{2}}
\end{align*}
since $|m(r,z,t)|^2\leq \rho_0(r,z)\chi(t)$ 
for almost all $(r,z,t)\in \Omega\times (0,T)$ and for any $m\in X$.\\
Now, let $m\in X$ be a point of continuity of $I$. By density of $X_0$ there exists $(m_k)_{k\in \mathbb{N}}\subset X_0$ for which $m_k\to m$ in $C([0,T],H_w(\Omega;\mathbb{R}^2))$. Due to Lemma \ref{perturbation property}, for any $k$ there exists a sequence $(m_{k_l})_{l\in \mathbb{N}}\subset X_0$ so that $m_{k_l}\to m_k$ in $C([0,T],H_w(\Omega))$ and
\begin{align*}
\|m_{k_l}\|_{L^2(\Omega\times [0,T])}^2&\geq \|{m_k}\|_{L^2(\Omega\times [0,T])}^2\\ &+ \beta \left(\int_{\Omega}\int_0^T \rho_0(r,z)\chi(t) dt dz dr - \|{m_k}\|_{L^2(\Omega\times [0,T])}^2\right)^{\frac{1}{2}}.
\end{align*}
Now, by a diagonal argument, for any $k\in \mathbb{N}$ there exists a $\tilde{m}_k\in X_0$ so that $\tilde{m}_k \to m$ in $C([0,T],H_w(\Omega))$ and
\begin{align*}
\|\tilde{m_k}\|_{L^2(\Omega\times [0,T])}^2&\geq \|{m_k}\|_{L^2(\Omega\times [0,T])}^2\\ &+ \beta \left(\int_{\Omega}\int_0^T \rho_0(r,z)\chi(t) dt dz dr - \|{m_k}\|_{L^2(\Omega\times [0,T])}^2\right)^{\frac{1}{2}}
\end{align*}
by Lemma \ref{perturbation property}.
As a consequence we have
\begin{align*}
\|m\|_{L^2(\Omega\times [0,T])}^2&\geq \|m\|_{L^2(\Omega\times [0,T])}^2\\ &+ \beta \left(\int_{\Omega}\int_0^T \rho_0(r,z)\chi(t) dt dz dr - \|m\|_{L^2(\Omega\times [0,T])}^2\right)^{\frac{1}{2}},
\end{align*}
which implies $\|m\|_{L^2(\Omega\times [0,T])}\geq\int_{\Omega}\int_0^T \rho_0(r,z) \chi(t) dt dz dr$. Lemma \ref{weak solution condition} finishes the proof.
\end{proof}
For the purpose of finding an initial subsolution, we will present a new construction. This construction is one of the main novelties of this work:
\begin{lem}\label{Konstruktion Sulsg}
Let $\rho_0, p\colon \Omega_{\mathbb{R}}\to \mathbb{R}$ be given as in (\ref{Dichte}), i.e.
\begin{align*}
\rho_0(r,z)=r,\quad p(\rho_0(r,z))=r^{\gamma}
\end{align*}
for $\gamma >1$. Then there exist $\tilde{m} \in C([0,T],C(\overline{\Omega};\mathbb{R}^2))\cap C([0,T],H_w(\Omega;\mathbb{R}^2))$ and $\tilde{U}\in C^1(\overline{\Omega}\times [0,T];S_0^2)$ such that for every $\chi\in C([0,T])$ satisfying $$\chi(t)> 2\|e(\rho_0(\cdot),\tilde{m}(\cdot,t),\tilde{U}(\cdot,t))\|_{L^{\infty}(\Omega_{\mathbb{R}})}$$ and for $$q_0(r,t)=p(\rho_0(r,z))+\frac{\chi(t)}{2},$$ it holds that 
\begin{align}
\partial_t \tilde{m}+\operatorname{div}(\tilde{U})+\nabla q_0=0 &\text{ in } \Omega_{\mathbb{R}}\times (0,T)\label{Sublsg},\\
\operatorname{div}(\tilde{m})=0&\text{ in } \Omega_{\mathbb{R}}\times (0,T),\label{Divergenzbedingung}\\
\tilde{m}_r=0 &\text{ on } \{r=\delta\}\cup \{r=R\}\label{slip rand bed},\\
e(\rho_0(r,z),\tilde{m}(r,z,t),\tilde{U}(r,z,t))<\frac{\chi(t)}{2} &\text{ in } \Omega_{\mathbb{R}}\times (0,T)\label{Energie}.
\end{align}

\end{lem}
\begin{proof}
We take the ansatz $\tilde{m}(r,z,t)=\begin{pmatrix}
0\\ \tilde{\chi}(t) \rho_0(r,z)\end{pmatrix}$ for a fixed function $\tilde{\chi}\in C^{\infty}([0,T],\mathbb{R}^{+})$. Note that this choice immediately gives $\tilde{m}\in C^\infty([0,T],H_w(\Omega;\mathbb{R}^2))\cap C^\infty([0,T],C^\infty(\overline{\Omega};\mathbb{R}^2))$. Moreover $\tilde{m}$ fulfils (\ref{Divergenzbedingung}) and (\ref{slip rand bed}). Then (\ref{Sublsg}) is equivalent to finding $\tilde{U}\colon \Omega_{\mathbb{R}}\times [0,T]\to S_0^2$ smooth so that
\begin{align*}
\begin{pmatrix}
\partial_r {\tilde{U}}_{rr}(r,z,t)+\partial_z {\tilde{U}}_{rz}(r,z,t)\\\partial_r {\tilde{U}}_{zr}(r,z,t)+ \partial_z {\tilde{U}}_{zz}(r,z,t)
\end{pmatrix} =\begin{pmatrix}
0\\-\tilde{\chi}'(t)\rho_0(r,z)
\end{pmatrix}+ \begin{pmatrix}
-\partial_r p(\rho_0(r,z))\\0
\end{pmatrix}.
\end{align*}
We choose ${\tilde{U}}_{rr}=-p(\rho_0)$, ${\tilde{U}}_{rz}={\tilde{U}}_{zr}=-\tilde{\chi}'\int_{0}^{r} \rho_0(s) ds$ and ${\tilde{U}}_{zz}=p(\rho_0)$. Then $\tilde{U}$ is a symmetric trace-free matrix solving
\begin{align*}
\partial_t \tilde{m}+ \operatorname{div}{\tilde{U}}+\nabla q_0=0
\end{align*}
in the sense of distributions. Moreover for $(r,z,t)\in \Omega_{\mathbb{R}}\times (0,T)$, we trivially have
\begin{align*}
e(\rho_0(r,z),\tilde{m}(r,z,t),\tilde{U}(r,z,t))\leq \|e(\rho_0(\cdot),\tilde{m}(\cdot,t),\tilde{U}(\cdot,t))\|_{\infty}<\frac{\chi(t)}{2}
\end{align*}
for all $\chi\in C^{\infty}([0,T],\mathbb{R}^+)$ with
\begin{align}\label{Ungleichung für chi}
\chi(t)> 2\|e(\rho_0(\cdot),\tilde{m}(\cdot,t),\tilde{U}(\cdot,t))\|_{L^{\infty}(\Omega_{\mathbb{R}})}.
\end{align}
Note the right-hand side admits a uniform bound in $t$, because $e$ is a continuous function composed with the bounded functions $\rho_0$, $\tilde m$, and $\tilde U$. Therefore,~\eqref{Ungleichung für chi} can be satisfied by suitable choice of $\chi$.

\end{proof}
The subsolutions we just constructed are not yet suitable inital subsolutions since they do not necessarily satsify the local energy inequality. In fact, ~\eqref{Ungleichung für chi} together with Lemma  \ref{properties of the energy}(ii) reveals that
\begin{align}\label{prohibition}
\int_{\Omega}|\tilde{m}(\cdot,0)|^2 dzdr \leq 2\int_{\Omega} \rho_0 e(\rho_0,\tilde{m}(\cdot,0),\tilde{U}(\cdot,0))dz dr < \int_{\Omega} \chi(0)\rho_0(r,z) dz dr.
\end{align}
We will show that if a subsolution $(m_0,U_0)$ satsifies the initial condition 
\begin{equation}\label{doubleCI}
|{m}_0(r,z,0)|^2=\rho_0(r,z)\chi(0)
\end{equation} 
for a.e.\ $(r,z)\in \Omega_{\mathbb{R}}$, then we can construct weak solutions which up to some time obey the local energy inequality. Unfortunately,~\eqref{prohibition} prohibits this property. To overcome this discrepancy, we will construct subsolutions $(m_0,U_0,q_0)$ fulfilling~\eqref{doubleCI} via convex integration. Since this construction requires a convex integration step to get suitable subsolutions and then again a convex integration to get from subsolutions to weak solutions, this mechanism is sometimes called {\em double convex integration}. Let us now construct the aforementioned suitable initial subsolutions:
\begin{lem}\label{Subsolution lemma}
Let $\rho_0$, $p$, $\chi$, and $(\tilde{m},\tilde{U},q_0)$ be as in Lemma \ref{Konstruktion Sulsg}. Then there exists $(m_0,U_0)$ so that
\begin{align*}
\begin{cases}
&\partial_t m_0+\operatorname{div}(U_0)+\nabla q_0=0\\
&\operatorname{div}(m_0)=0
\end{cases}
\end{align*} 
distributionally in $\Omega_{\mathbb{R}}\times (0,T)$ and $(m_0)_r=0$ on $\{r=\delta\}\cup \{r=R\}$ with the following properties:
\begin{enumerate}[(i)]
\item $(m_0,U_0,q_0)$ is continuous on $\Omega_{\mathbb{R}}\times (0,T]$,
\item $m_0\in C([0,T],H_w(\Omega;\mathbb{R}^2))$,
\item $e(\rho_0(r,z),m_0(r,z,t),U_0(r,z,t))<\frac{\chi(t)}{2}$ for all $(r,z,t)\in \Omega_{\mathbb{R}}\times (0,T]$,
\item $|m_0(r,z,0)|^2=\rho_0(r,z)\chi(0)$ for almost all $(r,z)\in \Omega_{\mathbb{R}}$\label{initial value condition}.
\end{enumerate}
\end{lem}
The proof is very similar to that of~\cite[Proposition 5]{convexintegration2} and is given in the appendix.

\begin{proof}[Proof of Theorem \ref{non-uniqueness for compressible Euler equations}]
Assume that $\rho_0$, $p_0$, $(m_0,U_0, q_0)$, and $\chi\in C^{\infty}([0,T],\mathbb{R}^{+})$ are given as in Lemma \ref{Subsolution lemma}. 
Then by Lemma \ref{Subsolution criterion}, there exist infinitely many weak solutions $(\rho,m)\in C^1(\Omega)\times C([0,T],H_w(\Omega;\mathbb{R}^2))$ of~\eqref{Sublsg kompressible Euler} with $\rho=\rho_0$ so that
\begin{align*}
\begin{cases}
m(r,z,0)=m_0(r,z,0) &\text{ for a.e. } (r,z)\in \Omega_{\mathbb{R}},\\
|m(r,z,t)|^2=\rho_0(r,z)\chi(t) &\text{ for a.e. } (r,z,t)\in\Omega_{\mathbb{R}}\times (0,T],\\
m_r=0 &\text{ on } \{r=\delta\}\cup \{r=R\}.
\end{cases}
\end{align*}
If we set $m^0(\cdot)=m_0(\cdot,0)$, the momentum field $m$ satisfies~\eqref{pointwise constraint subsolutions compressible Euler}.
\end{proof}
\begin{proof}[Proof of Theorem \ref{non-uniqueness for compressible Euler equations with energy inequality}]
As we have seen in Theorem \ref{non-uniqueness for compressible Euler equations}, there exist $m^0\in L^{\infty}(\Omega;\mathbb{R}^2)$ and infinitely many admissible weak solutions $m$ of~\eqref{Sublsg kompressible Euler} with $\rho=\rho_0$, where for continuous $\chi\colon \mathbb{R}\to \mathbb{R}^+$ with $\chi(t)>2 \| e(\rho_0(\cdot),\tilde{m}(\cdot,t),\tilde{U}(\cdot,t)\|_{L^{\infty}(\Omega)}$ and $\tilde{m}$, $\tilde{U}$ from Lemma \ref{Konstruktion Sulsg} we have
\begin{align}\label{Darstellung m ueber chi}
|m(r,z,t)|^2 &= \rho_0(r,z) \chi(t) &&\text{ a.e.\ in } \Omega_{\mathbb{R}}\times (0,T],\\
|m(r,z,0)|^2&=\rho_0(r,z)\chi(0) &&\text{ a.e.\ in } \Omega_{\mathbb{R}},\notag\\
m_r&=0 &&\text{ on } \{r=\delta\}\cup \{r=R\}.\notag
\end{align}
Now, we show that we can select weak solutions from Theorem \ref{non-uniqueness for compressible Euler equations} which satisfy the admissibility condition~\eqref{local energy inequality for incompressible Euler equations momentum}  if $\chi$ is chosen properly.  More precisely, we show that if $\chi$ satisfies $\chi(t)>2 \| e(\rho_0(\cdot),\tilde{m}(\cdot,t),\tilde{U}(\cdot,t)\|_{L^{\infty}(\Omega)}$ and solves a certain differential inequality, then weak solutions of~\eqref{Sublsg kompressible Euler} will obey~\eqref{local energy inequality for incompressible Euler equations momentum}.

As a starting point, we assume that $\chi>2 \| e(\rho_0,\tilde{m},\tilde{U})\|_{L^{\infty}(\Omega)}$ satisfies the following differential inequality:
\begin{align}\label{differential inequality}
\frac{1}{2}\chi'(t) \leq  -R^{\frac{1}{2}}\gamma\max\{R^{\gamma-2},\delta^{\gamma-2}\}\chi(t)^{\frac{1}{2}}-\frac{R^{\frac{1}{2}}}{2\delta^{2}}\chi(t)^{\frac{3}{2}}
\end{align}
with an initial value to be fixed at the end of this proof.
Recall that $\varepsilon(\rho_0)+\frac{p(\rho_0)}{\rho_0}=\frac{\gamma}{\gamma -1}r^{\gamma -1}$, which implies
\begin{align*}
\left|\nabla\left(\varepsilon(\rho_0)+\frac{p(\rho_0)}{\rho_0}\right)\right|\leq \gamma \max\{R^{\gamma-2},\delta^{\gamma-2}\}.
\end{align*}
Moreover we have 
\begin{align*}
\left|\nabla\left(\frac{1}{\rho_0}\right)\right|= \frac{1}{r^2}\leq \frac{1}{\delta^{2}}.
\end{align*}
Note that due to (\ref{Darstellung m ueber chi}) we have $|m|\leq \rho^{\frac{1}{2}}_0\chi^{\frac{1}{2}}$ a.e. in $\Omega\times [0,T]$.
Hence,~\eqref{differential inequality} entails the differential inequality 
\begin{align}\label{equivalent energy inequality}
\frac{1}{2}\chi'(t)\leq - m\cdot \nabla \left(\varepsilon(\rho_0)+ \frac{p(\rho_0)}{\rho_0}\right)- \frac{\chi(t)}{2}m \cdot\nabla \left(\frac{1}{\rho_0}\right).
\end{align}
Recall that $\rho_0=r$ is time-independent, so that a straightforward calculation shows that (\ref{equivalent energy inequality}) is equivalent to (\ref{local energy inequality for incompressible Euler equations momentum}). 
Now, for $\chi$ to satisfy~\eqref{differential inequality}, we demand $\chi$ to be the solution of 
\begin{align*}
\frac{1}{2}\chi'(t)=-R^{\frac{1}{2}}\gamma\max\{R^{\gamma-2},\delta^{\gamma-2}\}\chi(t)^{\frac{1}{2}}-\frac{R^{\frac{1}{2}}}{2\delta^{2}}\chi(t)^{\frac{3}{2}}
\end{align*}
with $\chi(0)=\chi^0$ chosen sufficiently large so that $\chi(t) >2 \| e(\rho_0(\cdot),\tilde{m}(\cdot,t),\tilde{U}(\cdot,t)\|_{L^{\infty}(\Omega)}$ up to some time $T>0$.
\end{proof}
\section{Symmetry breaking in the case of axisymmetry}
In this section we want to provide an example of axisymmetric swirl-free initial data for the $3D$ Euler equations for which the axisymmetry breaks down in the evolution. 
Similarly as before, for fixed $0<\delta<R$ we consider
\begin{align*}
\Omega\coloneqq \left\{(r,\theta,z):~\delta<r<R,~\theta \in \left[0,2\pi\right],~z\in \mathbb{T} \right\}.
\end{align*}
Let us recall the notion of subsolutions. In analogy to Section \ref{Section Axisymmetric swirl-free Euler equations vs isentropic compressible Euler equations} we introduce the space of solenoidal vector fields $H(\Omega;\mathbb{R}^3)$ as the completion of
\begin{align}
\{u\in C_c^{\infty}(\Omega;\mathbb{R}^3):\operatorname{div}(u)=0\}
\end{align}
with respect to the $L^2(\Omega;\mathbb{R}^3)$ topology. Again, $S_0^3$ is the space of symmetric trace-free $3\times 3$ matrices.
\begin{Defi}
A subsolution of the $3D$ incompressible axisymmetric swirl-free Euler equations with respect to the energy profile $\overline{e}$ is a triple $(v,U,q)\colon \Omega\times (0,T)\to \mathbb{R}^3\times S_0^3\times \mathbb{R}$, $v\in L^{\infty}((0,T),H(\Omega;\mathbb{R}^3))$ swirl-free and axisymmetric, $U\in L^1_{loc}(\Omega\times (0,T);S_0^3)$ and $q\in \mathcal{D}'(\Omega\times (0,T))$ such that
\begin{align}\label{Sublsg System incompressible Euler equations}
\begin{cases}
\partial_t v+ \operatorname{div}(U)+ \nabla q =0 &\text{ in } \Omega_{\mathbb{R}}\times (0,T)\\
\operatorname{div}(v) =0 &\text{ in } \Omega_{\mathbb{R}}\times (0,T)
\end{cases}
\end{align}
in the sense of distributions and 
\begin{align*}
\lambda_{\max}(v\otimes v - U )\leq \frac{2}{3}\overline{e},
\end{align*}
where $\lambda_{\max}$ denotes the largest eigenvalue.
\end{Defi}
\begin{rk}
In order to construct subsolutions, we will take the ansatz $U=U(r,z,t)$, $q=q(r,z,t)$, and $U_{r\theta}=U_{\theta r}=U_{\theta \theta}=U_{z\theta}=U_{\theta z}=0$, whereupon~\eqref{Sublsg System incompressible Euler equations} becomes equivalent to
\begin{align}\label{Sublsg in cylindrical coordinates}
\begin{cases}
\partial_t v_r + \partial_r U_{rr} + \frac{U_{rr}}{r} + \partial_z U_{rz}+\partial_r q=0\\
\partial_t v_z+ \partial_r U_{zr}+ \frac{U_{zr}}{r}+ \partial_z U_{zz}+\partial_zq=0\\
\partial_r(rv_r)+\partial_z(rv_z)=0.
\end{cases}
\end{align}
\end{rk}
Let us now give the main tool to get from subsolutions to weak solutions:
\begin{thm}\label{Subsolution criterion inscompressible Euler equations}
Let $\overline{e}\in L^{\infty}(\Omega\times (0,T))$ and let $(\overline{v},\overline{U},\overline{q})$ be a subsolution. Let $U\subset \Omega\times (0,T)$ be such that $(\overline{v},\overline{U},\overline{q})$ and $\overline{e}$ are continuous on $U$ and
\begin{align*}
\begin{cases}
e(\overline{v},\overline{U})<\overline{e} &\text{ in } U\\ 
e(\overline{v},\overline{U})=\overline{e} &\text{ a.e. in } (\Omega\times (0,T))\setminus U.
\end{cases}
\end{align*}
Then there exist infinitely many weak solutions $v\in L^{\infty}((0,T),H(\Omega;\mathbb{R}^3))$ of (\ref{Cauchy problem for the axisymmetric Euler equations}) such that
\begin{align*}
v=\overline{v} &\text{ a.e. in } (\Omega\times (0,T))\setminus U,\\
\frac{1}{2}|v|^2=\overline{e} &\text{ a.e. in } \Omega\times (0,T),\\
p=\overline{q}-\frac{2}{3}\overline{e} &\text{ a.e. in } \Omega\times (0,T).
\end{align*}
Moreover, if $\overline{v}(\cdot,t)\rightharpoonup v^0(\cdot)$ in $L^2(\Omega;\mathbb{R}^3)$ as $t\to 0$, then $v$ solves the Cauchy problem (\ref{Cauchy problem for the axisymmetric Euler equations}).
\end{thm}
\begin{proof}
A proof can be consulted from \cite[Proposition 3.3]{convexintegration2}.
\end{proof}
Let $0<\delta<r_0<R$ and define $\alpha_0\colon \Omega\to \mathbb{R}$
\begin{align*}
\alpha_0(r)=\begin{cases}
-\frac{1}{r} &\text{ if } r\in (\delta,r_0)\\
\frac{1}{r} &\text{ if } r\in (r_0,R).
\end{cases}
\end{align*}
Then we set
\begin{align}\label{initial data symmetry breaking}
v_0(x)=v_0(r,\theta,z)\coloneqq \begin{pmatrix}
0\\0\\\alpha_0(r)
\end{pmatrix},
\end{align}
where $(r,\theta,z)$ denote cylindrical coordinates. Obviously $v_0$ is divergence-free, axisymmetric and swirl-free.

Finally, we can prove Theorem~\ref{symmetry breaking introduction} in the following form:

\begin{thm}\label{symmetry breaking}
Let $v_0$ be as in~\eqref{initial data symmetry breaking}. Then, up to some time $T>0$, there exist infinitely many admissible weak solutions of~\eqref{Cauchy problem for the axisymmetric Euler equations} on $\mathbb{R}^3\times(0,T)$ which are not axisymmetric for any $t\in(0,T)$. 
\end{thm}
\begin{proof}
To find a suitable subsolution of~\eqref{Cauchy problem for the axisymmetric Euler equations}, we take the ansatz
\begin{align*}
\overline{v}(x,t)=\overline{v}(r,\theta,z,t)=\begin{pmatrix}
0\\0\\\alpha(r,t)
\end{pmatrix}, \quad \overline{U}(x,t)=\overline{U}(r,\theta,z,t)=\begin{pmatrix}
\beta(r,t)& 0 &\gamma(r,t)\\
0 & 0 &0\\
\gamma(r,t)& 0 &-\beta(r,t)
\end{pmatrix}
\end{align*}
for functions $\alpha,\beta,\gamma$, $\alpha(\cdot,0)=\alpha_0(\cdot)$ to be determined. Evidently such $v$ is axisymmetric and swirl-free and is divergence-free, i.e., it fulfils the third line of~\eqref{Sublsg in cylindrical coordinates}. Now, for this ansatz the first line of~\eqref{Sublsg in cylindrical coordinates} results in 
\begin{align*}
\partial_r \beta + \frac{\beta}{r}+ \partial_r q =0
\end{align*}
since $v_r=0$ and $\partial_z \gamma=0$. The latter is fulfilled if $q$ is chosen as
\begin{align*}
q(r)=\frac{1}{2}\alpha^2+\frac12\int_{1}^r \frac{\alpha(s)^2}{s}ds
\end{align*}
and $\beta=-\frac{1}{2}\alpha^2$. In order to deal with the second line of~\eqref{Sublsg in cylindrical coordinates}, we set $\alpha(r,t)=\frac{f(r,t)}{r}$ for $f=f(r,t)$. Moreover, note that $\partial_r(r\gamma)=\gamma + r\partial_r \gamma$. Then the second line of~\eqref{Sublsg in cylindrical coordinates} can be seen to be equivalent to
\begin{align*}
\partial_t \alpha(r,t)+ \partial_r \gamma(r,t) + \frac{\gamma(r,t)}{r}=0,
\end{align*}
and the latter is equivalent to
\begin{align*}
\partial_t {f(r,t)}+ \partial_r (r\gamma(r,t))=0.
\end{align*}
Now, for $\lambda>0$ small we set 
\begin{align}\label{Definition of gamma}
\gamma(r,t)=-\frac{\lambda}{2r}(1-f(r,t)^2)=-\frac{\lambda}{2}\left(\frac{1}{r}-r\alpha(r,t)^2\right).
\end{align}
This choice of $\gamma$ in (\ref{Definition of gamma}) implies $\partial_r (r\gamma)=\frac{\lambda}{2}\partial_r f^2$.
Hence, as in~\cite{SZEKELYHIDI20111063,BardosSzWiedemann}, we end up with a Burgers equation
\begin{align}\label{Burgers equation for f}
\partial_t f + \frac{\lambda}{2}\partial_r f^2=0
\end{align}
with initial data
\begin{align}\label{initial condition for Burgers equation for f}
f(r,0)=r{\alpha(r,0)}=\begin{cases} -1 &\text{ if } r\in (\delta,r_0)\\
1 &\text{ if } r\in (r_0,R).
\end{cases}
\end{align}
The Burgers equation~\eqref{Burgers equation for f}, \eqref{initial condition for Burgers equation for f} admits a rarefaction solution which takes the form
\begin{align}\label{Definition of f}
f(r,t)=\begin{cases}
-1 &\text{ if } r\in (\delta,r_0-\lambda t)\\
\frac{r-r_0}{\lambda t} &\text{ if } r\in (r_0-\lambda t, r_0+\lambda t)\\
1 &\text{ if } r\in (r_0+\lambda t, R).
\end{cases}
\end{align}
If we set $\alpha(r,t)=\frac{f(r,t)}{r}$ for $f$ defined by~\eqref{Definition of f} and $\gamma$ as in~\eqref{Definition of gamma}, the second line of~\eqref{Sublsg in cylindrical coordinates} is solved up to some positive time $T$. We now turn to the energy inequality. Let us compute $\overline{v}\otimes \overline{v} - \overline{U}$:
\begin{align*}
\overline{v}\otimes \overline{v} - \overline{U}=\begin{pmatrix}
-\beta &0 &-\gamma\\
0 &0 &0\\
-\gamma & 0 &\alpha^2 + \beta
\end{pmatrix}.
\end{align*} 
Then 
\begin{align*}
\det(\overline{v}\otimes \overline{v}-\overline{U}-\mu I_3)=
%-\mu \det\begin{pmatrix}
%-\beta-\mu &-\gamma\\
%-\gamma & \alpha^2+\beta-\mu
%\end{pmatrix}\\
%&=-\mu\left[(-\beta-\mu)(\alpha^2+\beta-\mu)-\gamma^2\right]\\
&=-\mu\left[\left(\frac{1}{2}\alpha^2-\mu\right)^2-\gamma^2\right].
\end{align*}
%
%since $\beta=\frac{1}{2}\alpha^2$ and $\alpha^2+\beta=\alpha^2-\frac{1}{2}\alpha^2=\frac{1}{2}\alpha^2$. 
This implies 

\begin{align*}
e(\overline{v},\overline{U})&=\frac{1}{2}\alpha^2+\frac{\lambda}{2}\left(\frac{1}{r}-r\alpha^2\right)\\
&=\frac{1}{2r^2}\left(1-(1-r\lambda)(1-f(r,t)^2)\right).
\end{align*}
Now, we set
\begin{align*}
\overline{e}(r,\theta,t)= \frac{1}{2r^2}\left(1-\frac{\varepsilon}{2}(1+\sin^2(\theta))(1-r\lambda)(1-f(r,t)^2)\right)
\end{align*}
for small $\varepsilon>0$. Moreover let
$$U\coloneqq \{(r,\theta,z,t)\in \Omega\times (0,T):~r_0-\lambda t < r <r_0+\lambda t\}.$$
Then, as $|f|=1$ outside $U$, we have
\begin{align*}
e(\overline{v},\overline{U})=\frac{1}{2r^2}=\overline{e}=\frac{1}{2}|v_0|^2 \text{ in } (\Omega\times (0,T))\setminus U.
\end{align*}
In addition, for $\lambda$ and  $\varepsilon$ sufficiently small, we have
\begin{align*}
e(\overline{v},\overline{U})&=\frac{1}{2r^2}\left(1-(1-r\lambda)(1-f(r,t)^2)\right)\\
&\leq  \frac{1}{2r^2}\left(1-\frac{\varepsilon}{2}(1+\sin^2(\theta))(1-r\lambda)(1-f(r,t)^2)\right)\\
&\leq\frac{1}{2}|v_0|^2
\end{align*}
in $U$ because $|f|\leq 1$. More precisely, since $|f|<1$ we have 
\begin{align*}
e(\overline{v},\overline{U})&< \overline{e} \text{ in } U\\
e(\overline{v},\overline{U})&=\overline{e}=\frac{1}{2}|v_0|^2 \text{ in } (\Omega\times (0,T))\setminus U.
\end{align*}
Then by Theorem \ref{Subsolution criterion inscompressible Euler equations} there exist infinitely many weak solutions $v\in L^{\infty}((0,T),H(\Omega;\mathbb{R}^3))$ with $\frac{1}{2}|v|^2=\overline{e}$ a.e. in $\Omega\times (0,T)$ and $v(\cdot,0)=v_0$. This implies admissibility, as for $t>0$ we have
\begin{align*}
\frac{1}{2}\int_{\Omega} |v(r,\theta,z,t)|^2 dz d\theta dr =\int_{\Omega} \overline{e}(r,\theta,z,t) dz d\theta dr <\frac{1}{2}\int_{\Omega} |v_0|^2 dz d\theta dr.
\end{align*}
Finally, since the energy $\bar e$ depends on the angular coordinate $\theta$ for all positive times, we see that the axisymmetry is instantaneously lost.
\end{proof}
%
%
%\begin{rk}
%Theorem \ref{symmetry breaking} shows that for given axisymmetric swirl-free initial data there exist infinitely many weak solutions $v$ which depend on $\theta$ and hence are not axisymmetric. This shows that the axisymmetry breaks down. 
%%\textcolor{red}{We don't know whether the swirl-freeness can break down as well?}
%\end{rk}
%
%
\begin{rk}
In the case of axisymmetric data with swirl, one can easily adapt the construction in~\cite{BardosSzWiedemann} to find an example of instantaneous symmetry breaking: Indeed, by simply adding a zero vertical component, we obtain such an example for a pure-swirl initial velocity. 
\end{rk}

\section*{Appendix}

\subsection*{Equivalence of Notions of Weak Solution}
We give a detailed and explicit computation to justify~\eqref{polartransf}, which in turn serves to verify Remark~\ref{Remark weak solutions 3D and axisymmetric}.

For the nonlinear term we have
\begin{align*}
v\otimes v :\nabla_x \varphi= ~&v_x^2\partial_x \varphi_x + v_xv_y\partial_y \varphi_x + v_x v_z \partial_z \varphi_x + v_y v_x \partial_x \varphi_y + v_y^2 \partial_y \varphi_y + v_y v_z \partial_z \varphi_y\\
&+ v_zv_x \partial_x \varphi_z + v_z v_y \partial_y \varphi_z + v_z^2 \partial_z \varphi_z.
\end{align*}
Now, we get
\allowdisplaybreaks
\begin{align*}
v_x^2\partial_x\varphi_x=v_r^2\cos^2(\theta)
&\Big(\cos^2(\theta)\partial_r \varphi_r-\cos(\theta)\sin(\theta)\partial_r \varphi_{\theta}-\frac{\sin(\theta)\cos(\theta)}{r}\partial_{\theta}\varphi_r\\
&+\frac{\sin^2(\theta)}{r}\varphi_r+\frac{\sin^2(\theta)}{r}\partial_{\theta}\varphi_{\theta}+\frac{\sin(\theta)\cos(\theta)}{r}\varphi_{\theta}\Big),\\
v_xv_y\partial_y\varphi_x=v_r^2\cos(\theta)
\sin(\theta)
&\Big(\cos(\theta)\sin(\theta)\partial_r \varphi_r-\sin^2(\theta)\partial_r \varphi_{\theta}+\frac{\cos^2(\theta)}{r}\partial_{\theta}\varphi_r\\
&-\frac{\sin(\theta)\cos(\theta)}{r}\varphi_r-\frac{\cos(\theta)\sin(\theta)}{r}\partial_{\theta}\varphi_{\theta}-\frac{\cos^2(\theta)}{r}\varphi_{\theta}\Big),\\
v_xv_z\partial_z \varphi_x=v_r\cos(\theta)v_z&\Big(\partial_z\varphi_r\cos(\theta)-\partial_z\varphi_{\theta}\sin(\theta)\Big),\\
v_yv_x\partial_x\varphi_y=v_r^2\cos(\theta)
\sin(\theta)
&\Big(\cos(\theta)\sin(\theta)\partial_r \varphi_r+\cos^2(\theta)\partial_r \varphi_{\theta}-\frac{\sin^2(\theta)}{r}\partial_{\theta}\varphi_r\\
&-\frac{\sin(\theta)\cos(\theta)}{r}\varphi_r-\frac{\cos(\theta)\sin(\theta)}{r}\partial_{\theta}\varphi_{\theta}+\frac{\sin^2(\theta)}{r}\varphi_{\theta}\Big),\\
v_y^2\partial_x\varphi_x=v_r^2\sin^2(\theta)
&\Big(\sin^2(\theta)\partial_r \varphi_r+\cos(\theta)\sin(\theta)\partial_r \varphi_{\theta}+\frac{\sin(\theta)\cos(\theta)}{r}\partial_{\theta}\varphi_r\\
&+\frac{\cos^2(\theta)}{r}\varphi_r+\frac{\cos^2(\theta)}{r}\partial_{\theta}\varphi_{\theta}-\frac{\sin(\theta)\cos(\theta)}{r}\varphi_{\theta}\Big),
\end{align*}
\begin{align*}
v_yv_z\partial_z \varphi_y=v_r\cos(\theta)v_z&\Big(\partial_z
\varphi_r\sin(\theta)+\partial_z\varphi_{\theta}\cos(\theta)\Big),\\
v_zv_x\partial_x\varphi_z=v_r\cos(\theta)v_z
&(\cos(\theta)\partial_r-\sin(\theta)\partial_{\theta})\varphi_z,\\
v_zv_y\partial_y\varphi_z=v_r\sin(\theta)v_z
&(\sin(\theta)\partial_r+\cos(\theta)\partial_{\theta})\varphi_z.
\end{align*}
Hence we get
\begin{align*}
v\otimes v:\nabla_x\varphi&=v_r^2\Big(\cos^4(\theta)+2\cos^2(\theta)\sin^2(\theta)+\sin^4(\theta)\Big)\partial_r\varphi_r+ v_rv_z(\partial_z \varphi_r+ \partial_r \varphi_z) + v_z^2\partial_z\varphi_z\\
&=v_r^2\partial_r\varphi_r + v_rv_z\partial_z\varphi_r+v_rv_z \partial_r \varphi_z+v_z^2\partial_z\varphi_z
\end{align*}
as desired.

\subsection*{Convex Integration Lemmata}

\begin{proof}[Proof of Lemma~\ref{weak solution condition}]
Let $m\in X$ be such that
\begin{align*}
|m(r,z,t)|^2=\rho_0(r,z)\chi(t) \text{ for a.e. } (r,z,t)\in \Omega_{\mathbb{R}}\times (0,T).
\end{align*}
By density of $X_0$ there is a sequence $(m_k)_{k\in \mathbb{N}}$ so that $m_k\to m$ in $(X,d)$. For any $m_k$, let $U_k$ be the associated smooth matrix. Since
\begin{align}\label{lokale Energiegleichung fuer die Folge}
e(\rho_0(r,z),m_k(r,z,t),U_k(r,z,t))<\frac{\chi(t)}{n},
\end{align}
Lemma \ref{properties of the energy} implies
\begin{align*}
|U_k(r,z,t)|\leq e(\rho_0(r,z),m_k(r,z,t),U_k(r,z,t))<\frac{\chi(t)}{2}
\end{align*}
for all $(r,z,t)\in \Omega_{\mathbb{R}}\times (0,T)$ and hence $\|U_k\|\leq \frac{\chi(t)}{2}.$
As a consequence there exists $U\in L^{\infty}(\Omega\times (0,T))$ so that along a subsequence we have
\begin{align*}
U_k\mathrel{\ensurestackMath{\stackon[1pt]{\rightharpoonup}{\scriptstyle\ast}}} U \text{ in } L^{\infty}(\Omega\times (0,T)).
\end{align*}
Due to Lemma \ref{properties of the energy} we know that $\overline{\text{hint}K_{\rho_0,\chi}^{co}}$ is convex and compact and hence $m\in X$ with associated matrix field $U$ solves (\ref{Sublsg system}).
Now, it follows by (\ref{lokale Energiegleichung fuer die Folge}) that $(m,U,q_0)$ takes values in $K_{\rho,\chi}^{co}$ almost everywhere whence by Lemma \ref{properties of the energy}(v) we have $(m,U,q_0)\in K_{\rho_0,\chi}$ a.e. in $\Omega_{\mathbb{R}}\times (0,T)$. Then Lemma \ref{subsolutions to weak solutions} implies that $(\rho_0,m)$ is a weak solution of (\ref{compressible Euler equations in terms of the momentum on the strip}) in $\Omega\times (0,T)$.
\end{proof}
\begin{proof}[Proof of Lemma~\ref{bairelem}]
For $\varepsilon >0$ let $\eta_{\varepsilon}(r,z,t)=\frac{1}{\varepsilon^3}\eta\left(\frac{r}{\varepsilon},\frac{z}{\varepsilon},\frac{t}{\varepsilon}\right)$ where $\eta$ is the standard mollifier.
Let $m\in X$ and $(m_k)_{k\in \mathbb{N}}\subset X$ such that $m_k\to m$ in $(X,d)$. Note that for any $\varepsilon >0$ Young's inequality implies
\begin{align}\label{erste Ungleichung}
m_k\ast \eta_{\varepsilon} \to m\ast \eta_{\varepsilon} \text{ in } L^2([0,T],H(\Omega;\mathbb{R}^2)) \text{ as } k\to \infty.
\end{align}
Moreover, by the properties of the convolution we have
\begin{align}\label{zweite Ungleichung}
m\ast \eta_{\varepsilon} \to m \text{ in } L^2([0,T],H(\Omega;\mathbb{R}^2)) \text{ as } \varepsilon\to 0.
\end{align}
Now, define $I_{\varepsilon}\colon (X,d)\to L^2([0,T],H(\Omega;\mathbb{R}^2))$, $m\mapsto m\ast \eta_{\varepsilon}$. Then (\ref{erste Ungleichung}) shows that $I_{\varepsilon}$ is continuous and (\ref{zweite Ungleichung}) establishes pointwise convergence of $I_{\varepsilon}$ to $I$. Hence $I$ is a Baire-1 map. Therefore the set of points of continuity of $I$ is residual in $(X,d)$.
\end{proof}

\begin{proof}[Proof of Lemma~\ref{Subsolution lemma}]
Let $\tilde{m}$ and $\tilde{U}$ be as in Lemma \ref{Konstruktion Sulsg} and consider $\tilde{X}_0$ to be the set of continuous momentum fields $m$ with associated $U\colon \Omega_{\mathbb{R}}\times [0,T)\to S_0^2$ so that 
\begin{align*}
\begin{cases}
\partial_t m + \operatorname{div}(U)+ \nabla q_0=0 \text{ in } \Omega\times (0,T),\\
\operatorname{div}(m)=0,\\
\end{cases}
\end{align*}
\begin{align}
&e(\rho_0(\cdot),m(\cdot,t),U(\cdot,t))<\frac{\chi(t)}{2} \text{ for all } t\in (0,T],\label{energy inequality in the construction}\\
&\text{supp}(m-\tilde{m})\subset \Omega\times \left[0,\frac{T}{2}\right),\\
&m_r=0 \text{ on } \{r=\delta\}\cup\{r=R\},\\
&U=\tilde{U} \text{ in } \Omega \times \left[\frac{T}{2},T\right).
\end{align}
Note that $\tilde{X}_0$ is contained in a bounded set $B\subset H(\Omega)$, by virtue of~\eqref{energy inequality in the construction}. Denote by $\tilde{X}$ the closure of $\tilde{X}_0$ with respect to a metrization $d$ of the convergence in $C([0,T],H_w(\Omega;\mathbb{R}^2))$.

The key tool for this construction is the following claim, which can be obtained by minor modifications of the perturbation property, Lemma~\ref{perturbation property}:

\textbf{Claim}: Let $\varnothing \neq \Omega_0\Subset \Omega$ be given. For any $\alpha >0$ there exists $\beta >0$ so that for all $(m,U)$ with $m\in \tilde{X}_0$ and
\begin{align*}
\int_{\Omega_0} |m(r,z,0)|^2-\rho_0(r,z) \chi(0) dz dr <-\alpha,
\end{align*}
there exists a sequence $((m_k,U_k))_{k\in \mathbb{N}}$ fulfilling
\begin{enumerate}[(i)]
\item $m_k\in \tilde{X}_0$ for all $k\in \mathbb{N}$,
\item $\text{supp}(m_k-m)\subset \Omega_0\times [0,\delta_k]$ for a sequence $(\delta_k)_{k\in \mathbb{N}}$, $\delta_k\to 0$ to be fixed,
\item $m_k\to m$ in $(\tilde{X},d)$,
\item $\liminf_{k\to \infty}\int_{\Omega_0} |m_k(r,z,0)|^2 dz dr\geq \int_{\Omega_0} |m(r,z,0)|^2 dz dr + \beta \alpha^2$.
\end{enumerate}
Now, let $\Omega_1\Subset \Omega$ and set $m_1(r,z,t)=\tilde{m}(r,z,t)$ and $U_1(r,z,t)=\tilde{U}(r,z,t)$ for all $(r,z,t)\in \Omega\times [0,T)$. Then by~\eqref{energy inequality in the construction} we have
\begin{equation}
\begin{aligned}\label{perturbation property für m1}
\int_{\Omega_1} |m_1(r,z,0)|^2 dz dr &\leq \int_{\Omega_1} 2\rho_0(r,z)e(\rho_0(r,z),m_1(r,z,0),U_1(r,z,0))dz dr\\
&< \int_{\Omega_1} \rho_0(r,z) \chi(0) dz dr,
\end{aligned}
\end{equation}
hence 
\begin{align*}
\alpha_1\coloneqq -\int_{\Omega_1}\left(|m_1(r,z,0)|^2-\rho_0(r,z)\chi(0)dz dr\right)>0.
\end{align*}
Applying the claim to $\Omega_1$ and $\delta=\frac{T}{2}$, we get $m_2\in \tilde{X}_0$ so that
\begin{align*}
\text{supp}(m_2-m_1,U_2-U_1)\subset \Omega_1\times \left[0,\frac{T}{2}\right],\\
\int_{\Omega_1} |m_2(r,z,0)|^2dz dr \geq \int_{\Omega_1}|m_1(r,z,0)|^2 +  \beta\alpha_1^2.
\end{align*}
Now, for $k\geq 2$ consider $\Omega_k\supset \Omega_1$ with $\cup_{k=1}^{\infty} \Omega_k =\Omega$, $\Omega_k\subset \Omega_{k+1}\Subset \Omega$ and $|\Omega_{k+1}\setminus \Omega_k|\leq 2^{-k}$ for all $k\in \mathbb{N}$. 
Assume we have constructed $(m_3,U_3),\dots, (m_n,U_n)$ with $m_k\in \tilde{X}_0$ for $k=3,\dots,n$ by using the above claim. Again we observe by (\ref{energy inequality in the construction})
\begin{align*}
\alpha_k=-\int_{\Omega_k} |m_k(r,z,0)|^2-\rho_0(r,z)\chi(0) dz dr>0. 
\end{align*}
The claim applied to $\Omega_k$, $\delta_k=\frac{T}{2^k}$ and $\alpha_k$ yields $m_{k+1}\in \tilde{X}_0$ and $U_{k+1}$ with
\begin{align}
\text{supp}(m_{k+1}-m_k,U_{k+1}-U_k)\subset \Omega_{k}\times [0,2^{-k}T], \label{Support Bedingung}\\
d(m_{k+1},m_k)<2^{-k} \label{Konvergenz in der (X,d)},\\
\int_{\Omega_k} |m_{k+1}(r,z,0)|^2 dz dr \geq \int_{\Omega_k} |m_k(r,z,0)|^2 dz dr + \beta \alpha_k^2\label{Perturbation der Folge}.
\end{align}
Recalling that $(\tilde{X},d)$ is a complete metric space, it follows from~\eqref{Konvergenz in der (X,d)} that there exists $m_0\in C([0,T],H_w(\Omega;\mathbb{R}^2))$ so that $m_k\to m_0$ in $(\tilde{X},d)$. Taking (\ref{Support Bedingung}) into account, on any compact subset the sequence is eventually constant. More precisely:
\begin{align*}
\forall C\subset \Omega\times (0,T) \text{ compact}~\exists~k_0\in \mathbb{N}:~\forall k \geq k_0:~(m_k,U_k)=(m_{k_0},U_{k_0}) \text{ on } C.
\end{align*}
As a direct consequence we observe
\begin{align*}
(m_k,U_k)\to (m_0,U_0) \text{ in } C_{loc}(\Omega\times (0,T)),
\end{align*}
where $(m_0,U_0)$ satisfies $e(\rho_0(r,z),m_0(r,z,t),U_0(r,z,t))<\frac{\chi(t)}{2}$ in $\Omega_{\mathbb{R}}\times (0,T)$ and solves
\begin{align*}
\partial_t m_0+\operatorname{div} U_0 +\nabla q_0= 0,\\
\operatorname{div} (m_0) =0
\end{align*}
in $\Omega_{\mathbb{R}}\times (0,T)$. Now, it only remains to show that $|m_0(r,z,0)|^2=\rho_0(r,z)\chi(0)$. Note that (\ref{Perturbation der Folge}) is equivalent to
\begin{align*}
\alpha_{k+1} + \beta \alpha_k^2 \leq \alpha_k+ \int_{\Omega_{k+1}\setminus \Omega_k} \rho_0(r,z) \chi(0) dz dr, 
\end{align*}
which implies $|\alpha_k-\alpha_{k+1}|\leq C2^{-k}$ and hence $\alpha_k\to 0$.
Furthermore we have
\begin{align*}
0&>\int_{\Omega} |m_k(r,z,0)|^2-\rho_0(r,z) \chi(0) dz dr
\\&\geq \int_{\Omega_k} |m_k(r,z,0)|^2-\rho_0(r,z) \chi(0) dz dr + \int_{\Omega\setminus \Omega_k} |m_k(r,z,0)|^2-\rho_0(r,z)\chi(0) dz dr\\
&\geq -\alpha_k-C|\Omega\setminus \Omega_{k}|\\
&\geq -(\alpha_k + C 2^{-k-1})
\end{align*}
since $|\Omega\setminus \Omega_k|=\sum_{j=k}^{\infty} 2^{-j}=2^{-k-1}$. We conclude
\begin{align}\label{Konvergenz L^2 Konv der mk}
\lim_{k\to \infty}\int_{\Omega} |m_k(r,z,0)|^2 -\rho_0(r,z)\chi(0)dz dr = 0.
\end{align}
We denote by $\eta_{\varepsilon}$ a standard mollifier. Combining~\eqref{Konvergenz in der (X,d)} with the fact that $m_k\in C([0,T],H_w(\Omega;\mathbb{R}^2))$ for any $k\in \mathbb{N}$, we may choose $\gamma_k<2^{-k}$ so that 
\begin{align}
\|\left((m_k-m_{k+1})\ast \eta_{\gamma_l}\right)(\cdot,0)\|_{L^2(\Omega)}&<2^{-k} \text{ for all } l \leq k\label{Faltung von mk - mk+1},\\
\sup_{t\in [0,T]} \| m_k-m_k\ast \eta_{\gamma_k}\|_{L^2(\Omega)}&<2^{-k}\label{mk - Faltung mk}.
\end{align}
As a consequence of (\ref{Faltung von mk - mk+1}) we deduce
\begin{align}\label{erste Faltungsabschaetzung}
\|((m_k-m_0)\ast \eta_{\gamma_k})(\cdot,0)\|_{L^2(\Omega)}&\leq \sum_{j=0}^{\infty} \|((m_{k+j}-m_{k+j+1})\ast \eta_{\gamma_k})(\cdot,0)\|_{L^2(\Omega)}\\
&\leq \sum_{j=0}^{\infty} 2^{-(k+j)}=2^{-k+1}\notag.
\end{align}
Thanks to the estimates (\ref{Faltung von mk - mk+1}), (\ref{mk - Faltung mk}) we end up with
\begin{align*}
\|(m_k-m_0)(\cdot,0)\|_{L^2(\Omega)}&\leq \|(m_k-m_k\ast \eta_{\gamma_k})(\cdot,0)\|_{L^2(\Omega)}\\
&+ \|((m_k-m_0)\ast \eta_{\gamma_k})(\cdot,0)\|_{L^2(\Omega)}\\
&+ \|(m_0\ast \eta_{\gamma_k}-m_0)(\cdot,0)\|_{L^2(\Omega)}\\
&\leq 2^{-k}+ 2^{-k+1}+ \|(m_0\ast \eta_{\gamma_k}-m_0)(\cdot,0)\|_{L^2(\Omega)}.
\end{align*}
Since $m_0\in C([0,T],H_w(\Omega;\mathbb{R}^2))$ this now implies $m_k(\cdot,0)\to m_0(\cdot,0)$ in $L^2(\Omega;\mathbb{R}^2)$ and because of (\ref{Konvergenz L^2 Konv der mk}) we have $|m_0(r,z,0)|^2=\rho_0(r,z)\chi(0)$ for almost all $(r,z)\in \Omega$.
\end{proof}

\printbibliography
\end{document}